\documentclass[pdflatex,sn-mathphys-num]{sn-jnl}% Math and Physical Sciences Numbered Reference Style
\usepackage{graphicx}%
\usepackage{multirow}%
\usepackage{amsmath,amssymb,amsfonts}%
\usepackage{amsthm}%
\usepackage{mathrsfs}%

\DeclareFontFamily{U}{rsfs}{\skewchar\font127}
\DeclareFontShape{U}{rsfs}{m}{n}{%
  <-6> rsfs5
  <6-8> rsfs7
  <8-> rsfs10
}{}

\usepackage[title]{appendix}%
\usepackage{xcolor}%
\usepackage{textcomp}%
\usepackage{manyfoot}%
\usepackage{booktabs}%
\usepackage{algorithm}%
\usepackage{algorithmicx}%
\usepackage{algpseudocode}%
\usepackage{listings}%
\usepackage{bm}
\renewcommand{\theequation}{\thesection.\arabic{equation}}

\makeatletter
\newcommand{\subalign}[1]{%
  \vcenter{%
    \Let@ \restore@math@cr \default@tag
    \baselineskip\fontdimen10 \scriptfont\tw@
    \advance\baselineskip\fontdimen12 \scriptfont\tw@
    \lineskip\thr@@\fontdimen8 \scriptfont\thr@@
    \lineskiplimit\lineskip
    \ialign{\hfil$\m@th\scriptstyle##$&$\m@th\scriptstyle{}##$\hfil\crcr
      #1\crcr
    }%
  }%
}
\makeatother

\theoremstyle{thmstyleone}%
\newtheorem{theo}{Theorem}%  meant for continuous numbers
\newtheorem{pro}[theo]{Proposition}% 

\theoremstyle{thmstyletwo}%
\newtheorem{exa}{Example}%
\newtheorem{rem}{Remark}%

\theoremstyle{thmstylethree}%
\newtheorem{defin}{Definition}%
\newtheorem{lem}{Lemma}%
\newtheorem{cor}{Corollary}%

\begin{document}

\title[Article Title]{Comparison of Dirichlet forms for stable-like random walks on groups of polynomial volume growth}


\author*[1]{\fnm{Laurent} \sur{Saloff-Coste}}\email{lsc@math.cornell.edu}

\author[2]{\fnm{Ruoqi} \sur{Zhang}}\email{rz433@cornell.edu}
\equalcont{These authors contributed equally to this work.}

\affil*[1]{\orgdiv{Math Department}, \orgname{Cornell University}, \orgaddress{\street{212 Garden Ave}, \city{Ithaca}, \postcode{14853}, \state{NY}, \country{USA}}}

\affil[2]{\orgdiv{Math Department}, \orgname{Cornell University}, \orgaddress{\street{212 Garden Ave}, \city{Ithaca}, \postcode{14853}, \state{NY}, \country{USA}}}


\abstract{In earlier works, we studied natural examples of stable-like random walks on finitely generated nilpotent groups and, more generally, on groups of polynomial volume growth. These walks are driven by measures of radial type, coordinate-wise type, or convex combinations of such measures. In the present work, we explore when two such random walks have comparable behavior, as measured by the equivalence of their associated Dirichlet forms. We develop criteria for this equivalence in terms of both geometric and algebraic features of the driving measures and the group.}




\maketitle

\section{Introduction} In a number of previous works \cite{CKSCWZ,CKSCWZ1,SCRZ,SCZ-nil}, large families of symmetric random walks driven by probability measures featuring stable-like jumps on finitely generated nilpotent groups and groups having  polynomial volume growth have been considered and studied. The goal of the present work is to understand when two such measures have equivalent Dirichlet forms.
Following  ideas from the  works cited above,  to each of the stable-like measures mentioned above on a finitely generated nilpotent group, we associate a certain quasi-norm on the group as well as a descending central series of subgroups. We describe
a partition of this family of  measures into classes having equivalent Dirichlet forms where the classes in question are defined in terms of the associated quasi-norms or, equivalently, in terms of the associated descending series. In other words, these classes of measures can be equivalently described using the geometric concepts of quasi-norms, or the algebraic notion of central descending series and any two measures in the same class have equivalent Dirichlet forms. It is likely that measures from two distinct classes have non-equivalent Dirichlet forms but we have not been able to prove this except in special simple cases. We also explain how this classification works on finitely generated groups of polynomial volume growth in which case it is essential and unavoidable to use  the fact that the groups of polynomial volume growth are exactly those which contain a nilpotent subgroup with finite quotient (Gromov's theorem on finitely generated groups of polynomial volume growth).

Throughout this work we use the notation $f_1\asymp f_2$ between to nonnegative functions to indicate  there are constants $0<c,C<+\infty$ such that $cf_1\le f_2\le Cf_1$ on the common domain of these functions. We use $\preccurlyeq, \succcurlyeq$ for the related order relations.

\section{Key definitions}

\subsection{Quasi-norms} On a group $G$, a function $\|\cdot\|: G\mapsto [0,+\infty)$ is called a quasi-norm if it satisfies $\|g\|=0$ if and only if $g=e$, $\|g\|=\|g^{-1}\|$, and there exists a constant $C$ such that 
$\|gh\|\le C(\|g\|+\|h\|)$ for all $g,h\in G$. A quasi-norm is a norm when $C=1$.

 To any generating tuple $S=(s_1,\dots,s_k)\in G^k$ and weight tuple $\mathbf v=(v_1,\dots,v_k)\in (0,+\infty)^k$, we associate the quasi-norm
 \begin{equation}\label{QN}
     \|g\|_{S,\mathbf v}= \inf \left\{ \max_{s_i \in S}\left\{ \mbox{deg}_{s_i}(\theta)^{1/v_i} \right\}:  \theta = g \in G, \theta \in \bigcup_{m=0}^\infty [S\cup S^{-1}]^m\right\} \end{equation}
Here, $S\cup S^{-1}$ is interpreted as a finite formal alphabet, $\theta$ is a finite word over this alphabet, the equation $\theta=g$ is evaluation in $G$, and 
$$\mbox{deg}_{s_i}(\theta)=\#\{j: \theta_j\in \{s_i,s_i^{-1}\}\}$$ where $\theta=\theta_1\dots \theta_m.$ This definition is taken from \cite{SCZ-nil}. 

When $v_i=v$, $1\le i\le k$, we obviously have $$|g|^{1/v}_{S\cup S^{-1}}\le \|g\|_{S,\mathbf v}\le k|g|^{1/v}_{S\cup S^{-1}}$$ where $|\cdot|_{S\cup S^{-1}}$ is the word-length based on the symmetric generating set $S\cup S^{-1}$. That is $\|\cdot\|_{S,\mathbf v}\asymp |\cdot|^{1/v}_{S\cup S^{-1}}$ in this case. Note that $g\mapsto \|g\|_{S,\mathbf v}$
is a norm (a quasi-norm with $C=1$) if $v_i\ge 1$ for all $1\le i\le k$. We will mostly be interested in the case when $v_i>1/2$ when relating this definition to what we call stable-like measures.

\subsection{Stable-like measures}
 The family of stable-like symmetric probability measures of interest to us comprise two subfamilies and the convex combinations  of the union of the two families.  The first subfamily is the family of what we call {\em quasi-norm radial measures}. Let $G$ be equipped with a quasi-norm $\|\cdot\|=\|\cdot\|_{S,\mathbf v}$ as above with $v_i>1/2$ for $1\le i\le k$. Assume that
$$\#\{g\in G: \|g\|\le r\}\asymp r^\gamma$$ 
for some $\gamma\in (0,+\infty)$. We call {\em radial measure} associated to the quasi-norm $\|\cdot\|$ any symmetric probability measure $\rho=\rho_{S,\mathbf v}$  satisfying
 \begin{equation}\label{radial} \rho(g)\asymp (1+\|g\|)^{-1-\gamma}.
\end{equation}

As a simple example, assume that $G$ has polynomial volume growth of exponent $d$, that is,
$\#\{g:|g|_{S\cup S^{-1}}\le r\}\asymp r^d$, and consider the case when $\mathbf v=(v,\cdots,v)$ with $v=1/\alpha>1/2$. Then
$\|\cdot \|_{S,\mathbf v}\asymp |\cdot |^\alpha_{S\cup S^{-1}}$ and $\{g: \|g\|_{S,\mathbf v}\le r\}\asymp r^{\gamma}$ with $\gamma=d/\alpha$. In this case,
$$\rho_{S,\mathbf v}(g)\asymp (1+|g|_{S\cup S^{-1}})^{-\alpha-d}.$$

The second subfamily is the family of {\em coordinate-wise stable-like measures} defined as follow.
For any tuple of group elements (not necessarily generating) $S=(s_1,\ldots,s_k)$, consider the map $\pi_S$ that sends $ \bar{a}=(a_1,\ldots,a_k) \in \mathbb{Z}^k$ to $s_{1}^{a_1}\ldots s_{k}^{a_k} \in G$.  Given an exponent tuple $\bm{\alpha} = (\alpha_1,\ldots, \alpha_k) \in (0,2)^k$ (these are, in some sense, stability exponents), consider the symmetric probability measures $\nu_{S, \alpha}$  giving to any $h\in G$ a probability $\nu_{S,\bm{\alpha}}(h)$ proportional to (by convention, empty sums are equal to $0$)
\begin{align}
\label{cwsl}      
     \sum\limits_{ \bar{a}: 
      \pi_S(\bar{a})\in \{h,h^{-1}\} } \frac{1}{\left(1+ \sum_{i=1}^k |a_i|^{\alpha_i}\right)^{1 + \sum_{i=1}^k \frac{1}{\alpha_i}}} . 
\end{align}

\begin{exa}
  Let $N$ denote the group of \(3\times 3\) upper unitriangular matrices and let
$M_{ij}$ denotes the elementary matrix with a 1 on the diagonal and 
in the $(i,j)$-entry, and   zeros elsewhere. Take $S = (M_{23}, M_{12}, M_{13})$ and $\bm{\alpha} = (\alpha_1, \alpha_2,\alpha_3) \in (0,2)^3$. 
Then 
\begin{align*}
    \nu_{S,\bm{\alpha}}\left( \begin{pmatrix}
1 & a & c \\
0 & 1 & b \\
0 & 0 & 1
\end{pmatrix}\right) =  &\left( |b|^{\alpha_1} + |a|^{\alpha_2} + |c|^{\alpha_3} \right)^{1+ \sum_i \frac{1}{\alpha_i}} \\
& +  \left( |-b|^{\alpha_1} + |-a|^{\alpha_2} + |ab-c|^{\alpha_3} \right)^{1+ \sum_i \frac{1}{\alpha_i}} 
\end{align*}

\end{exa}

In addition of these two types of measures (radial and coordinate wise stable-like), we want to consider finite linear combinations of such measures, $$\nu=\sum_1^\ell c_i\nu_i, \;\;c_i>0,\;\; \sum_1^\ell c_i=1,$$ where each component $\nu_i$ is supported on  a subgroup $H_i$ of $G$ and is either radial on $H_i$ as in (\ref{radial}) or a coordinate-wise stable-like measure based on a generating tuple of $H_i$ as in (\ref{cwsl}). We assume that $G=\langle H_1,\dots, H_k\rangle$ so that the support of $\nu$ generates $G$. The simplest such convex combinations are the $1d$-singular stable-like measures
\begin{equation}\label{1d}\mu_{S,\bm{\alpha}}(g) \asymp \frac{1}{k}\sum_{i=1}^k \sum_{a\in \mathbb{Z}} \frac{c_{\alpha_i}\mathbf 1_{ s_i^a}(g)}{(1+|a|)^{1+\alpha_i}}.    
\end{equation}
Here, each component is the push-forward of the $1d$-stable-like symmetric probability measure 
$$c_{\alpha_i}(1+|a|)^{-1-\alpha_i} \asymp (1+|a|^{\alpha_i})^{-1-1/\alpha_i}, a\in \mathbb Z,$$ onto the discrete one-parameter subgroup $<s_i>$ of $ G$. We assume that $G=\langle s_1,\dots, s_k\rangle$. The class of $1d$-singular stable-like measures was introduced and studied explicitly in \cite{SCZ-nil}. The ideas and results of \cite{SCZ-nil} play an important role in the present work.

Given any symmetric probability measure $\mu$ on a (countable) group $G$, the associated Dirichlet form is
$$\mathcal E_\nu (f,f)=\frac{1}{2}\sum_{g,h\in G}|f(gh)-f(g)|^2\nu(h).$$

Our goal is to break down the collection of stable-like measure described above into subsets in which for any two measures $\mathcal E_{\mu_1}\asymp \mathcal E_{\mu_2}$. Ideally, we would like to describe the equivalence classes of our family of measures under the equivalence relation $\mu_1\sim \mu_2$ if and only if $\mathcal E_{\mu_1}\asymp \mathcal E_{\mu_2}$. It is likely that the partition we provide is made of those equivalence classes but this is not proved here.

\section{On a nilpotent group $N$}

In this section, our underlying group is a finitely generated nilpotent group $N$. The constructions described in this section {\em do not} translate straightforwardly to the case of groups of polynomial volume growth which will be discuss in a later section.

\subsection{Descending series and volume growth}
\label{sec: series}
Given a tuple $S=(s_1,\dots,s_k)\in N^k$ (viewed, together with their inverses, as letters in a formal alphabet $S\cup S^{-1}$), 
let $\mathcal{C}(S)$ be the collection of all formal commutators on the alphabet $S\cup S^{-1}.$  Starting from a weight tuple $\mathbf v=(v_1,\dots,v_k)$, define the weight system $w$ on $\mathcal{C}(S)$ inductively by 
$$  w(s_i) = v_i \text{ and }  w ([c_1,c_2] )= w(c_1) +  w (c_2)$$
Denote by $w_1< w_2 < w_3 < \ldots $ the sequence  of distinct values taken by $w(c)$ where $c$ runs over all formal commutators in $\mathcal{C}(S)$. The collection $\mathcal{C}(S)$ can be partitioned as 
$$\mathcal{C}(S) = \bigcup_{i=1}^\infty \mathcal{C}_{w_i}(S)  \text{ where } \mathcal{C}_{w_i}(S) := \{c \in \mathcal{C}(N): w(c) = w_i\} $$
For $i \ge 1$ define the subgroup (of course, here we use evaluation in $N$)
$$N_{w_i} =  \left\langle \bigcup_{j=i}^\infty \mathcal{C}_{w_j}(S)  \right \rangle =   \left\langle c \in  \mathcal{C}(S): w_N(c) \ge w_i\right \rangle \subseteq N.$$
As $N$ is nilpotent, there exists a minimal $j_*$ with the property that $N_{w_{j_* + 1}}= \{e\}$ and the  largest subgroup $N_1$ is always equal to $N$ as $S$ is a set of generators of $N$. 
These subgroups form a (central) series (see \cite{SCZ-nil})
$$ \mathcal{N}_{S,\mathbf{v}} := \left\{N =  N_{w_1} \supseteq \ldots \supseteq N_{w_{j_*}}   \supset N_{w_{j_* + 1}}  = \{e\}\right\}$$
    called \textit{the weighted series induced by $(S,\mathbf v)$}. It satisfies the relations
$$[N_{w_i},N_{w_j}]\subseteq N_{w_{i} + w_j}$$
which we refer to as the additive filtration property.  Set 
\begin{align}
\label{eq: gamma}
    \gamma =\gamma(N,S,\mathbf v)=\gamma_{S,\mathbf v}= \sum_{i=1}^{j_*} \mbox{rank}(N_{w_i}/N_{w_{i+1}}) w_i  
\end{align}
where $\mbox{rank}(N_{w_i}/N_{w_{i+1}})$ is the torsion-free rank of the abelian group $N_{w_i}/N_{w_{i+1}}$. The following known result motivates the definition of $\gamma$ and is important to the definition of radial stable like measures.

 \begin{theo}[{\cite[Theorem 3.2.1]{CKSCWZ}}, {\cite[Th. 3.2 and Rm. 3.3]{SCZ-nil}}]
\label{thm:norm_comp}  Referring to the setting and notation introduced above, for all $r \ge 1$, 
    $$\left|\{g \in N: \|g\|_{S,\mathbf v} \le r\}\right|  \asymp r^{\gamma}$$
\end{theo}

Define $ \mathfrak{w}: N \to [0, \infty)$ by   
\begin{align*}
    \mathfrak{w}(g) =\left\{
    \begin{array}{l}
        \infty \text{ if } \langle g\rangle \text{ is finite},\\
        \underset{i=1,\ldots,j_*}{\max}\{w_i: \langle g\rangle \cap N_{w_i} \mbox{ is infinite}\} \text{ otherwise.}
    \end{array}\right.
\end{align*}

\begin{rem}
\label{rem: norm}
    The definition of $\mathfrak{w}$ is compatible with \cite[Def. 2.13 and Prop. 2.17]{SCZ-nil} when $g$ is not torsion. Namely, \cite[Definition 2.13]{SCZ-nil} introduces $j_{\mathfrak w}(g)$  and  we have (for any non-torsion element in $N$), $\mathfrak{w}(g)=w_{j_{\mathfrak{w}(g)}}$. Moreover, \cite[Proposition 2.17]{SCZ-nil} (together with other results in \cite{SCZ-nil}) states that for each fixed $g\in N$ and any $n\in \mathbb Z$, $\|g^n\|\asymp_g n^{1/\mathfrak{w}(g)}$. This explain the importance of $\mathfrak w$: for each $g\in N$ $1/\mathfrak w(g)$ is the power exponent of growth of the quasi-norm of $g^n$ as $n$ tends to infinity. 
\end{rem}

\begin{rem}
All the constructions/definitions above  (i.e. the quasi-norm, descending series and weight function $\mathfrak{w}$ and the parameter $\gamma$) depend upon the pair of tuples $(S,\mathbf v)$.  When useful, we will note this dependence using subscripts.
\end{rem}

Following the notation in \cite{SCZ-nil}, 
we define 
$$\mbox{core}(S,\mathbf{v}) = \{s_i \in S\;:\; v_i = \mathfrak{w}(s_i), i=1,\ldots, k\}$$
In \cite[Proposition 2.14]{SCZ-nil}, it's proved that for $1\le j \le j_*$, any formal commutator in $\mathcal{C}_{w_j}(S)$ whose image in $N_{w_j}$ is free in $N_{w_j}/N_{w_{j+1}}$,  must only use letters in $\mbox{core}(S,\mathbf{v})$; the collection of such formal commutators is denoted by 
$\mathcal{C}_{w_j}(\mbox{core}(S,\mathbf{v}))$ and furthermore, we denote
$$\mathcal{C}(\mbox{core}(S,\mathbf{v})) := \bigcup_{j=1}^{j_*}\mathcal{C}_{w_j}(\mbox{core}(S,\mathbf{v}))$$

\begin{exa}
\label{exa: m1}
     Take $N$ to be the group of dimension-4 upper unitriangular matrices. Let $M_{ij}$ be the matrix with all entries set to  zero except for a 1 in the $(i,j)$ position and along the diagonal.
Let 
$$S = (M_{12}, (M_{12})^2, M_{23}, M_{34}, M_{13}, M_{24}, M_{14}) \text{ and }\mathbf{v} = (1,2,2,2,3, 3, 4)$$
Then 
$$N_{S,\mathbf{v}} = \{N_{w_1} \supseteq N_{w_2} \supseteq N_{w_3} \supseteq N_{w_4} \supseteq N_{w_5} \supset  N_{w_6} \supsetneq \{e\}\}$$
where 
\begin{align*}
    w_1 = 1; & \quad N_{w_1} = N\\
    w_2 = 2; & \quad N_{w_2}=\langle(M_{12})^2, M_{23}, M_{34}, M_{13}, M_{24}, M_{14} \rangle \\  
    w_3 =  3; & \quad N_{w_3}= \langle M_{13}, M_{24}, M_{14}\rangle \\
    w_4 = 4; &  \quad N_{w_4}=\langle  (M_{13})^2, M_{24}, M_{14}\rangle \\
    w_5 = 5; &  \quad N_{w_5} =\langle M_{14} \rangle \\
    w_6 = 6; &  \quad N_{w_5} =\langle (M_{14})^2 \rangle 
\end{align*}
The growth exponent $\gamma$ defined in Theorem \ref{thm:norm_comp} can be calculated as
$$\gamma = 2* 3 + 4*2 + 6*1 =  20 $$
and 
$$\mbox{core}(S,\mathbf{v}) = \{(M_{12})^2, M_{23}, M_{34}\}.$$
Moreover, given 
$$g = \begin{pmatrix}
1 & a_{12} & a_{13} & a_{14} \\
0 & 1      & a_{23} & a_{24} \\
0 & 0      & 1      & a_{34} \\
0 & 0      & 0      & 1
\end{pmatrix}$$
it can be written as 
\begin{align*}
    g & = M_{14}^{a_{14}}M_{24}^{a_{24}} M_{34}^{a_{34}} M_{13}^{a_{13}}M_{23}^{a_{23}}M_{12}^{a_{12}} \\
    &= \left([M_{12}^{2x_{14}}, [M_{23}^{x_{14}}, M_{34}^{x_{14}}]]M_{14}^{\epsilon_{14}} \right)\left([M_{23}^{x_{24}}, M_{34}^{x_{24}}] M_{24}^{\epsilon_{24}} \right) \\
        & \quad \cdot \left(M_{34}^{a_{34}}  \right) \left([M_{12}^{2x_{13}}, M_{23}^{x_{13}}]M_{13}^{\epsilon_{13}}\right)
\left(M_{23}^{\epsilon_{23}}\right)\left(M_{12}^{x_{12}}M_{12}^{\epsilon_{12}}\right)
\end{align*}
where 
\begin{align*}
x_{14} &= \lfloor \sqrt[3]{a_{14}/2}\rfloor \text{ and }\epsilon_{14} = a_{14} - 2x_{14}^3\\
x_{24} & =\lfloor  \sqrt{a_{24}}\rfloor  \text{ and }\epsilon_{24} = a_{24} -  x_{24}^2 \\
x_{13} & = \lfloor  \sqrt{a_{13}/2}\rfloor  \text{ and }\epsilon_{13} = a_{13} -  2x_{13}^2 \\  
x_{12} & = \lfloor  a_{12}/2 \rfloor \text{ and }\epsilon_{12} = a_{12} - x_{12}
\end{align*}
So 
$$\|g\|_{S,\mathbf{v}} \asymp |a_{14}|^\frac{1}{6} + |a_{24}|^\frac{1}{4} + |a_{34}|^\frac{1}{2} + |a_{13}|^\frac{1}{4} + |a_{23}|^\frac{1}{2} + |a_{12}|^\frac{1}{2}$$
\end{exa}

We will show that the following useful result follows from \cite{SCZ-nil}.
\begin{pro} \label{frakeq}Given two pairs of tuples $(S,\mathbf v)$ and $(\tilde{S},\tilde{\mathbf v})$, we have
$$\|\cdot\|_{S,\mathbf v}\asymp \|\cdot\|_{\tilde{S},\tilde{\mathbf v}
} $$  if and only if
$$\mathfrak w_{S,\mathbf v}(s)=\mathfrak w_{\tilde{S},\tilde{\mathbf v}}(s)
\mbox{ for all } s\in \mbox{core}(S,\mathbf{v}) \cup \mbox{core}(\tilde{S}, \tilde{\mathbf{v}}).$$
\end{pro}

\begin{exa}
\label{exa: m2}
    Referring to the notations in Example \ref{exa: m1}, one can verify that if $(\tilde{S}, \tilde{\mathbf{v}}) = ((M_{12}, M_{23}, M_{34}), (2,2,2))$, $(S,\mathbf{v})$ and $(\tilde{S}, \tilde{\mathbf{v}})$ satisfy the equivalence hypothesis in Theorem \ref{frakeq}.  
\end{exa}

The following useful lemma illustrates that checking equality on the cores is the natural choice. 
\begin{lem}
\label{lem: commutators_wt} We have 
$$\mathfrak w_{S,\mathbf v}(s)=\mathfrak w_{\tilde{S},\tilde{\mathbf v}}(s)
\mbox{ for all } s\in \mbox{core}(S,v) \cup \mbox{core}(\tilde{S}, \tilde{v})$$
if and only if 
$$\mathfrak w_{S,\mathbf v}(c)=\mathfrak w_{\tilde{S},\tilde{\mathbf v}}(c)
\mbox{ for all } c\in\mathcal{C}(\mbox{core}(S,v)) \cup \mathcal{C}(\mbox{core}(\tilde{S}, \tilde{v})) $$
\end{lem}
\begin{proof}
    Equality on $\mathcal{C}(\mbox{core}(S,\mathbf{v})) \cup \mathcal{C}(\mbox{core}(\tilde{S}, \tilde{\mathbf{v}}))$ obviously imply equality on $\mbox{core}(S,\mathbf{v}) \cup \mbox{core}(\tilde{S}, \tilde{\mathbf{v}})$.  Assume now that $$\mathfrak w_{S,\mathbf v}(s)=\mathfrak w_{\tilde{S},\tilde{\mathbf v}}(s)
\mbox{ for all } s\in \mbox{core}(S,\mathbf{v}) \cup \mbox{core}(\tilde{S}, \tilde{\mathbf{v}}).$$ Take
     $c = [s_{i_1}, \ldots, s_{i_l}] \in \mathcal{C}_{w}(\mbox{core}(S,v))$; in particular, $$\mathfrak{w}^{S,\mathbf v}(c) = w = 
     \mathfrak{w}^{S,\mathbf v}(s_{i_1})  + \ldots + \mathfrak{w}^{S,\mathbf v}(s_{i_l}) = v_{i_1} + \ldots + v_{i_l}  $$ The hypothesis gives
     $\mathfrak{w}^{\tilde{S},\tilde{\mathbf v}}(c) \ge \mathfrak{w}^{S,\mathbf v}(c) $. Assume for the sake of contradiction that $\tilde{w} := \mathfrak{w}^{\tilde{S},\tilde{\mathbf v}}(c) > \mathfrak{w}^{S,\mathbf v}(c) $. That means there exists some $n \in \mathbb{N}$ such that $c^n $ can be written as a word over 
$\mathcal{C}_{\tilde{w}}(\tilde{S}, \tilde{\mathbf{v}})$ and moreover, there exist $m \in \mathbb{N}$ such that such that all letter in the word for $c^n$ that are torsion modulo deeper terms of the series are eliminated  i.e. $c^{nm} \in \langle \mathcal{C}_{\tilde{w}}(\mbox{core}(\tilde{S}, \tilde{\mathbf{v}}) )\rangle$. However, by the hypothesis, it follows $\mathfrak{w}^{S,\mathbf v}(c) \ge \tilde{w}$, a contradiction. 
\end{proof}

In fact, either of the two (equivalent) hypotheses in Lemma \ref{lem: commutators_wt}  will force the two weight functions to agree on all of $N$. Indeed, once one of these hypotheses holds, Proposition  \ref{frakeq} implies that,
    for any $n \in \mathbb{Z}$,
    $$\|g^n\|_{S,\mathbf v} \asymp_g n^\frac{1}{\mathfrak{w}_{S,\mathbf v}(g)}\asymp_g  n^\frac{1}{\mathfrak{w}_{\tilde{S},\tilde{\mathbf v}}(g) }\asymp_g   \|g^n \|_{\tilde{S},\tilde{\mathbf v}}$$
and hence, $\mathfrak{w}_{S,\mathbf v}(g)  = \mathfrak{w}_{\tilde{S},\tilde{\mathbf v}}(g) $.  
\begin{proof} (of Theorem \ref{frakeq})
    Suppose 
$\|\cdot\|_{S,\mathbf v} \asymp  \|\cdot\|_{\tilde{S},\tilde{\mathbf v}}$ and take $s \in \mbox{core}(S,v) \cup \mbox{core}(\tilde{S}, \tilde{v})$. In particular, $s$ satisfies that $\mathfrak{w}_{S,\mathbf v}(s)$ and $\mathfrak{w}_{\tilde{S},\tilde{\mathbf v}}(s)$ are both finite. 
By Remark \ref{rem: norm}, it follows that  for any $n \in \mathbb{Z}$, 
$$\|s^n\|_{S,\mathbf v} \asymp_s n^\frac{1}{\mathfrak{w}_{S,\mathbf v}(s)}\asymp_s  n^\frac{1}{\mathfrak{w}_{\tilde{S},\tilde{\mathbf v}}(s) }\asymp_s   \|s^n \|_{\tilde{S},\tilde{\mathbf v}}$$
This forces  $\mathfrak{w}_{S,\mathbf v}(s) = \mathfrak{w}_{\tilde{S},\tilde{\mathbf v}}(s)$ as desired. 
Now suppose the weight functions coincide on the cores. 
Take $g \in N$ and suppose $\|g\|_{S,\mathbf v} = R$. 
By \cite{SCZ-nil}, we can rewrite $g$ as follows 
$$g = 
\prod_{c \in \mathcal{C}(S)}
c^{x_c}, \qquad x_c   \preccurlyeq 
\begin{cases}
    R^{\mathfrak{w}_{S,\mathbf{v}}(c)} &\text{if }c \in \mathcal{C}(\mbox{core}(S,v)) \\
    1 & \text{otherwise}
\end{cases} $$
Again by Remark \ref{rem: norm} and  Remark \ref{lem: commutators_wt}, 
\begin{align*}
\|g\|_{\tilde{S},\tilde{\mathbf v}}  \asymp 
\sum_{c \in \mathcal{C}(S)}  
|x_c|^\frac{1}{\mathfrak{w}_{\tilde{S},\tilde{\mathbf v}} (c)}    \preccurlyeq  \sum_{c \in \mathcal{C}(\mbox{core}(S,v)) } 
R^\frac{\mathfrak{w}_{S,\mathbf{v}}(c)}{\mathfrak{w}_{\tilde{S},\tilde{\mathbf v}} (c)} 
\preccurlyeq R 
\end{align*}
The rest follows from symmetry. 
 \end{proof}

\subsection{Isolators and reduced descending series}
\label{sec: reduced}

For any subgroup $K$ of $N$, we define the isolator to be 
$$I(K) : = \{g \in K\;: \; \exists n \in \mathbb{Z}, g^n \in K \}$$
As $N$ is a finitely generated nilpotent group, by \cite[Theorem 4.5]{Hall_Isolator}, $I(K)$ is a subgroup and is a finite-index augmentation of $K$.
Given the weighted series $ \mathcal{N}_{S, \mathbf v}$ associated to $(S,\mathbf{v})$, we can construct 
\begin{align}
\label{series: isolator}
    I(N_{w_1}) \supseteq 
\ldots \supseteq I(N_{w_{j_*}})   \supset I(N_{w_{j_* + 1}})  = I(\{e\}) = \mbox{Tor}(N)
\end{align}
\begin{theo}
\label{thm: isolator}
    The series consisting of isolator subgroups has the following properties. 
    \begin{enumerate}[(i)]
        \item $I(N_{w_{i}}) /I(N_{w_{i+1}})$ is torsion-free. 
        \item The series is still central. 
        \item The image of  
        $\mathcal{C}_{w_i}(\mbox{core}(S,\mathbf{v}))$ 
       generates a finite-index subgroup of 
       $I(N_{w_{i}}) /I(N_{w_{i+1}})$. 
    \end{enumerate}
\end{theo}
\begin{proof}
To see (i), note that if $(gI(N_{w_{i+1}}))^m = e$, $g^m \in I(N_{w_{i+1}})$ and hence $g \in I(N_{w_{i+1}})$. 
For (ii), we verify  $[N, I(N_{w_i})] \subseteq I(N_{w_{i+1}})$. Take $g \in N$ and $x \in I(N_{w_i})$ with $x^m \in N_{w_i}$ for some $m \in \mathbb{Z}$. Note that 
$$[g, x^m] = [g,x]^m \cdot \xi, \quad \xi \in N_{w_{i+1}}$$
and because the original series is central, $[g,x^m] \in N_{w_{i+1}}$, some power of $[g,x]$ lies inside $N_{w_{i+1}}$ as desired. 
For (iii), recall that $\mathcal{C}_{w_i}(\mbox{core}(S,\mathbf{v}))$ generates the subgroup of $N_{w_i}$ that's non-torsion after modding out $N_{w_{i+1}}$. Equivalently, its image generates $$\{x I(N_{w_{i+1}}) : x\in N_{w_i}\} = (N_{w_i} \cdot I(N_{w_{i+1}}))/ I(N_{w_{i+1}}) 
$$
which is a finite-index subgroup of $I(N_{w_i})/I(N_{w_{i+1}})$. Indeed, the index is equal to $|I(N_{w_i}): (N_{w_i} \cdot I(N_{w_{i+1}}))|$, wich is bounded by $|I(N_{w_i}):N_{w_i}| < \infty$. 
\end{proof}

For any subsequence in Series \ref{series: isolator} containing identical subgroups, retain the last occurrence while removing all preceding duplicates to obtain 
\begin{align*}
 \mbox{Red}( \mathcal{N}_{S, \mathbf v}) = \Big\{ I(N_{w_{l_1}})\supset I(N_{w_{l_2}}) \supset\ldots \supset I(N_{w_{l_{m}}}) \supset   I(\{e\})  = \mbox{Tor}(N)\Big\}
\end{align*}

The series and the tuple of remaining weights
$ \mathbf{w}_{S,\mathbf v} = (w_{l_1},\ldots, w_{l_m})$
are called, respectively, the \textit{reduce series} and the \textit{reduced weight tuple} associated to $(S,\mathbf v)$.

\begin{rem}
\label{rem: reduced_wt}
    Note that weights in  $\mathbf{w}_{S,\mathbf v}$ are exactly those for which $\mathcal{C}_{w_{l_\bullet}}(\mbox{core}(S,\mathbf{v}))$ is nonempty; by Theorem \ref{thm: isolator} (iii), each successive quotient $I(N_{w_{l_\bullet}})/ I(N_{w_{l_{\bullet + 1}}}) $
     in $\mbox{Red}( \mathcal{N}_{S, \mathbf v})$ is generated by (the image of) $\mathcal{C}_{w_{l_\bullet}}(\mbox{core}(S,\mathbf v))$. Furthermore, the weight function  $ \mathfrak{w}_{S,\mathbf{v}}: N \to [0, \infty)$ can be defined equivalently as 
\begin{align*}
    \mathfrak{w}_{S,\mathbf{v}}(g) =\left\{
    \begin{array}{l}
        \infty \text{ if } \langle g\rangle \text{ is finite},\\
        \underset{i=1,\ldots,j_*}{\max}\{w_i: g\in I(N_{w_i})\}  =  \underset{w \in \mathbf{w}_{S,\mathbf v} }{\max}\{w: g\in I(N_{w})\} \text{ otherwise.}
    \end{array}\right.
\end{align*}
\end{rem}

\begin{rem}
\label{rem: reduced_se}
We record several key properties of  $\mbox{Red}(\mathcal{N}_{S,\mathbf{v}})$ and $\mathbf{w}_{S,\mathbf v}$. 
In $\mbox{Red}(\mathcal{N}_{S,\mathbf{v}})$, the largest subgroup $I(N_{w_{l_1}})$ need not to be with $N$; however, by Theorem \ref{thm: isolator}(i) and the pruning step in the construction of the reduced series, it is necessarily a finite-index subgroup of $N$. Furthermore, 
the tuple $\mathbf{w}_{S,\mathbf v}$ is closed in the sense that if $w_{l_i}, w_{l_j} \in \mathbf{w}_{S,\mathbf v}$ with 
    $w:= w_{l_i} + w_{l_j} \le w_{l_m} $, then $w_{l_i} + w_{l_j} \in \mathbf{w}_{S,\mathbf v}$.  Indeed, $w\in \{w_1,\ldots, w_{j_*}\}$ and 
    the hypothesis implies that $\mathcal{C}_{w_{l_i}}(\mbox{core}(S,\mathbf v))$ and $\mathcal{C}_{w_{l_j}}(\mbox{core}(S,\mathbf v))$ are nonempty and hence so is $\mathcal{C}_{w}(\mbox{core}(S,\mathbf v))$ and by Remark \ref{rem: reduced_wt}, $w \in \mathbf{w}_{S,\mathbf v}$. 
    The reduced series satisfies the additive filtration property as well, i.e.   
    $$[I(N_{w_{l_i}}), I(N_{w_{l_j}})] \subseteq I(N_w)$$
    Take $g_i \in I(N_{w_{l_i}})$ and $g_j \in I(N_{w_{l_j}})$. There exist $a,b \ge 1$ such that $g_i^a \in N_{w_{l_i}}$ and $g_j^b \in N_{w_{l_j}}$. 
    Then $[g_i, g_j]^{ab} = [g_i^a, g_j^b] \cdot \xi $ where $\xi$  lies in subgroups deeper in the series. Since $ [g_i^a, g_j^b] \in [N_{w_{l_i}}, N_{w_{l_j}}] \subseteq N_{w}$,     $[g_i,g_j] \in I(N_w)$ as desired.  
\end{rem}

\begin{theo}
\label{thm: wt_series} 
  Given two pairs of tuples $(S,\mathbf v)$ and $(\tilde{S},\tilde{\mathbf v})$, we have
$$\mathfrak w_{S,\mathbf v}(s)=\mathfrak w_{\tilde{S},\tilde{\mathbf v}}(s)
\mbox{ for all } s\in \mbox{core}(S,\mathbf{v}) \cup \mbox{core}(\tilde{S}, \tilde{\mathbf{v}})$$  
if and only if $\mbox{Red}(\mathcal{N}_{S, \mathbf v})$ and $\mbox{Red}(\mathcal{N}_{\tilde{S}, \tilde{\mathbf v}})$ are equal and the reduced weight tuples are equal, i.e. $\mathbf{w}_{S, \mathbf v} = \mathbf{w}_{\tilde{S}, \tilde{\mathbf v}}$. 
\end{theo}

\begin{proof}  
The only-if direction is immediate by the definition of the weight $\mathfrak{w}$. Suppose the weight functions coincide on the cores. The last subgroups in $\mbox{Red}(\mathcal{N}_{S, \mathbf v})$ and $\mbox{Red}(\mathcal{N}_{\tilde{S}, \tilde{\mathbf v}})$ are both equal to $\mbox{Tor}(N)$ and hence equal. Now consider  the last non-torsion subgroup, denoted by $I(N_w^{S, \mathbf v})$, in $\mbox{Red}(\mathcal{N}_{S, \mathbf v})$. By  Remark \ref{rem: reduced_wt}, $I(N_w^{S, \mathbf v})/\mbox{Tor}(N)$ is generated by the image of $\mathcal{C}_{w}(\mbox{core}(S,\mathbf{v}))$. By Lemma \ref{lem: commutators_wt}, the hypothesis implies $\mathfrak w_{S,\mathbf v}(c)=\mathfrak w_{\tilde{S},\tilde{\mathbf v}}(c)$ for any $c \in \mathcal{C}_{w}(\mbox{core}(S,\mathbf{v}))$ and therefore, $w \in \mathbf{w}_{\tilde{S}, \tilde{\mathbf v}}$ and $\mathcal{C}_{w}(\mbox{core}(S,\mathbf{v})) \subseteq I(N_w^{\tilde{S},\tilde{\mathbf{v}}})$. Moreover, 
$$I(N_w^{S, \mathbf v}) = \langle \mathcal{C}_{w}(\mbox{core}(S,\mathbf{v})), \mbox{Tor}(N)\rangle \subseteq I(N_w^{\tilde{S},\tilde{\mathbf{v}}})$$
The rest follows from symmetry and induction. 
\end{proof}

\begin{exa}
We continue with the example introduced in Example \ref{exa: m1}. In the notation of that example, the corresponding isolator series is given by
    \begin{align*}
        I(N_{w_1}) &=  I(N_{w_2}) = N \\
          I(N_{w_3}) &=   I(N_{w_4})= \langle M_{13}, M_{24}, M_{14}\rangle\\
    I(N_{w_5}) &=     I(N_{w_6}) =  \langle M_{14}\rangle
    \end{align*}
    and the corresponding reduced series is 
    $$\mbox{Red}(\mathcal{N}_{S,\mathbf{v}}) = \left\{  I(N_{w_2}) \supset   I(N_{w_4}) \supset    I(N_{w_6})   \right\}$$
Moreover, recall $(\tilde{S}, \tilde{\mathbf{v}})$ defined in Example \ref{exa: m2}. As shown there, one has
$$\mathfrak w_{S,\mathbf v}(s)=\mathfrak w_{\tilde{S},\tilde{\mathbf v}}(s)
\mbox{ for all } s\in \mbox{core}(S,\mathbf{v}) \cup \mbox{core}(\tilde{S}, \tilde{\mathbf{v}})$$  
One can further verify that these two tuples also satisfy the reduced-series hypothesis of Theorem \ref{thm: wt_series}. Indeed 
$$\mathcal{N}_{\tilde{S}, \tilde{\mathbf{v}}} =  \{N_{\tilde{w}_1} \supseteq N_{\tilde{w}_2} \supseteq N_{\tilde{w}_3} \supsetneq \{e\}\} $$
where 
\begin{align*}
    \tilde{w}_1 = 2; & \quad N_{\tilde{w}_1} = N\\
    \tilde{w}_2 = 4; & \quad N_{\tilde{w}_2}=\langle M_{13}, M_{24}, M_{14} \rangle \\  
    \tilde{w}_3 =  6; & \quad N_{\tilde{w}_3}= \langle M_{14}\rangle 
\end{align*}
Trivially, we have $\mbox{Red}(\mathcal{N}_{\tilde{S}, \tilde{\mathbf{v}}}) = \mathcal{N}_{\tilde{S}, \tilde{\mathbf{v}}}  = \mbox{Red}(\mathcal{N}_{S,\mathbf{v}})$ and $\mathbf{w}_{S, \mathbf v} = \mathbf{w}_{\tilde{S}, \tilde{\mathbf v}}$. 
\end{exa}

\subsection{A collection of stable-like measures}

\label{sec: convex}
We assume we are given a symmetric probability measure $\nu$ which is stable-like in the sense that it is a convex combination
\begin{equation}\nu=\sum _i^\ell c_i\nu_i \label{nu-conv}\end{equation}
of measures $\nu_i$ of the types (\ref{radial})-(\ref{cwsl}) supported on  given subgroups $H_i$, $1\le i\le \ell$, of $N$ and such that $N=<H_1,\dots,H_\ell>$. More specifically, for each $i\in \{1,\dots,\ell\}$, one of the following two possibility occurs:
\begin{enumerate}
    \item There is subgroup $H_i$ generated by a tuple $S_i=(s_{i,1},\dots,s_{i,k_i}) $ and there is a weight tuple $\mathbf v_i=(v_{i,1},\dots,v_{i,k_i})\in (1/2,+\infty)^{k_i}$ such that $\nu_i\asymp \rho_{S_i,\mathbf v_i}\asymp (1+\|\cdot\|_{S_i\mathbf v_i})^{-1-\gamma_i}$ where 
    $\#\{\|g\|_{S_i,\mathbf v_i}\le r\}\asymp r^{\gamma_i}$.  In this case, we define the tuple $\bm{\alpha}_i=(\alpha_{i,1},\dots,\alpha_{i,k_i})$ by setting $\alpha_{i,j}=1/{i,j}$, $1\le j\le k_i$.
    \item There is a subgroup $H_i$ generated by a tuple $S_i=(s_{i,1},\dots,s_{i,k_i}) $ and there is an exponent tuple $\bm{\alpha}_i=(\alpha_{i,1},\dots,\alpha_{i,k_i})\in (0,2)^{k_i}$ such
    $\nu_i \asymp \nu_{S_i,\bm {\alpha}_i}.$ In this case, we also define a weight tuple $\mathbf v_i$ by setting $v_{i,j}=1/\alpha_{i,j}$, $1\le j\le k_i$.    \end{enumerate}

Given this data, we construct a tuple $S=S_\nu$ of group elements, a weight tuple $\mathbf v=\mathbf v_\nu$ and an exponent tuple $\bm{\alpha}_\nu$, all of the same length $k=\sum_1^\ell k_i$, by concatenation of the tuples $S_i$, $\mathbf v_i$, and $\bm{\alpha}_i$ associated to $\mu_i$ on $H_i$, $1\le i\le \ell$. 
Having done this, we consider the 
 radial measure 
$$\rho_{S,\mathbf v}\asymp(1+\|\cdot\|_{S,\mathbf v})^{-1-\gamma}$$
where the quasi-norm $\|\cdot\|_\nu : =\|.\|_{S,\mathbf v}$ and $\gamma=\gamma_{\nu}=\gamma_{S_\nu,\mathbf v_\nu}$ are defined as in \eqref{QN} and \eqref{eq: gamma} respectively.  We also introduce the auxiliary $1d$-singular probability measure 
\begin{equation}\label{muS}\mu_{S,\bm{\alpha}}(g) \asymp \sum_{j=1}^k \sum_{a\in \mathbb{Z}} \frac{\mathbf{1}_{ s_j^a}(g)}{(1+|a|)^{1+\alpha_j}}.\end{equation}
In this definition, we use the notation $S=S_\nu=(s_1,\dots,s_k)$ to denote the tuple $S$ obtained by concatenation as described above.

 Our main result can be stated as follows.
\begin{theo}
\label{thm: main1}
Referring to the above setting and notation, given a symmetric stable-like probability measure $\nu$, we have
$$\mathcal E_\nu\asymp \mathcal E_{\rho_{S,\mathbf{v}}}\asymp \mathcal E_{\mu_{S,\bm{\alpha}}},$$
that is, the Dirichlet forms of $\nu$, of the radial measure $\rho_{S,\mathbf{v}}$ and of the $1d$-singular stable-like measure $\mu_{S,\bm{\alpha}}$, are all equivalent.
Moreover, if $\tilde{\nu}$ is another symmetric stable-like probability measure in our collection, form the tuple
$(\tilde{S}, \tilde{\mathbf{v}})$ by concatenating the tuples arising from its components, as before.
The following are equivalent 
\begin{enumerate}
    \item $\|\cdot\|_{\nu}\asymp \|\cdot \|_{\tilde{\nu}}$
     \item $\mathfrak w_{S,\mathbf{v}}(s)=\mathfrak w_{\tilde{S},\tilde{\mathbf{v}}}(s)$ for all $
s \in \mbox{\em core}(S,\mathbf{v}) \cup \mbox{\em core}(\tilde{S},\tilde{\mathbf{v}})$ 
    \item $\mathfrak w_{S,\mathbf{v}}(c)=\mathfrak w_{\tilde{S},\tilde{\mathbf{v}}}(c)$ for all $
 c\in\mathcal{C}(\mbox{\em core}(S,\mathbf{v})) \cup \mathcal{C}(\mbox{\em core}(\tilde{S},\tilde{\mathbf{v}}))$ 
\item  $\mbox{\em Red}(\mathcal{N}_{S,\mathbf{v}}) = \mbox{\em Red}(\mathcal{N}_{\tilde{S},\tilde{\mathbf{v}}})$  and $\mathbf{w}_{S,\mathbf{v}} = \mathbf{w}_{\tilde{S},\tilde{\mathbf{v}}} $
\end{enumerate}
If any of the equivalent conditions above holds, then  $\mathcal E_{\tilde{\nu}}\asymp\mathcal E_\nu$. 
\end{theo}

\begin{proof}
    The first part of the theorem follows directly from the Dirichlet forms comparison established in the next subsection. Indeed, the comparison $\mathcal{E}_{\rho_{S,\mathbf{v}}} \asymp \mathcal{E}_{\mu_{S,\bm{\alpha}}}$ follows directly from Theorem \ref{thm: rho_mu}, and Theorem \ref{thm: rho_mu} and Theorem \ref{thm: coord} together imply that, for $i=1,\ldots, \ell$, $$\mathcal{E}_{\nu_i} \asymp \mathcal{E}_{\mu_{S_i ,\bm{\alpha}_i}}$$
    and therefore
    $$\mathcal{E}_{\nu} = \sum_{i=1}^\ell  c_i \mathcal{E}_{\nu_i} \asymp  \sum_{i=1}^\ell c_i \mathcal{E}_{\mu_{S_i ,\bm{\alpha}_i}} \asymp \mathcal{E}_{\mu_{S,\bm{\alpha}}}. $$ 
     The second part is essentially a consolidation of the results surrounding  Proposition \ref{frakeq} and Lemma \ref{lem: commutators_wt} and Theorem \ref{thm: wt_series}. Once the norm equivalence in (1) is established, it is immediate that the corresponding Dirichlet forms are comparable.   
\end{proof}

\begin{cor}
    \label{lem: return_prob} 
    Referring to the notations above, 
        $$\nu^{(2n)}(e) \asymp \rho_\nu^{(2n)}(e) \asymp \mu_\nu^{(2n)}(e) \asymp n^{-\gamma}$$
        where $\gamma$ is as in {\em (\ref{eq: gamma})}.
\end{cor}
\begin{proof}
By {\cite[Theorem 5.1]{SCZ-nil}}, $\rho_\nu^{(2n)}(e) \asymp  n^{-\gamma}$. To extend  the return probability estimate for $\rho_\nu^{(2n)}$ to the other two measures, we use the following result from  \cite{PSCstab}: 
    if $\phi, \psi$ are two symmetric probability measures on a countable group $G$ such that $\mathcal{E}_\phi \le C\mathcal{E}_\psi$, 
    then $$\psi^{(2kn)}(e) \le 2 \phi^{(2n)}(e) + 2e^{-2kn}, k = [C] + 2$$ 
\end{proof}

\begin{exa}
\label{exa: m12}
     Take $N$ to be the group of dimension-4 upper unitriangular matrices. Let $M_{ij}$ be the matrix with all entries set to  zero except for a 1 in the $(i,j)$ position and along the diagonal as in Example \ref{exa: m1}.

Consider any stable-like measure on $N$ as in (\ref{nu-conv}) for which each each component $\nu_i$ the associated tuple of generators, $S_i$, only involves generators chosen among the matrices $M_{ij}$. For simplicity, assume that each matrix $M_{ij}$ appears at least once in the long tuples $S_\nu$. 

Now, consider the tuple $T=(M_{12},M_{23},M_{34},M_{13},M_{24},M_{14})$ and the exponent tuple $\boldsymbol{\beta}=(\beta_{12},\beta_{23},\beta_{34},\beta_{13},\beta_{24},\beta_{14})$, where $\beta_{ij}$ is the lowest of the exponents $\alpha$ attributed to $M_{ij}$ in the long tuple $\boldsymbol{\alpha}$. Let $\bf u$
 be the weight tuple $\bf u=1/\boldsymbol{\beta}$.  We have three probability measures associated with $(T,{\boldsymbol\beta})$:
The radial measure $\rho_{T,\bf u}$ defined at (\ref{radial}), the coordinate-wise stable measure defined at (\ref{cwsl}), $\nu_{T,\boldsymbol{\beta}}$, and the $1$-singular measure $\mu_{T,\boldsymbol{\beta}}$ defined at (\ref{muS}). Theorem \ref{thm: main1} tells us that $\nu$ and these three measures have equivalent Dirichlet forms. It also allows us to find the minimal  sub-tuples of $T$ and $\boldsymbol{\beta}$
that provide measures with equivalent Dirichlet forms.

These minimal sub-tuples are determined by inspecting
$\bf u=1/\boldsymbol{\beta}$ as follows:
\begin{enumerate}
\item Assume that $u_{13}\le u_{12}+u_{23}$ and $u_{24}\le u_{23}+u_{34}$. If $u_{14}\le u_{12}+u_{23}+u_{34}$, set $T^*=(M_{12},M_{23},M_{34})$
and ${\bf u}^*=(u_{12},u_{23},u_{34})$. If $u_{14}> u_{12}+u_{23}+u_{34}$, set $T^*=(M_{12},M_{23},M_{34},M_{14})$
and ${\bf u}^*=(u_{12},u_{23},u_{34},u_{14})$.
\item Assume that $u_{13}> u_{12}+u_{23}$ and $u_{24}\le u_{23}+u_{34}$. If $u_{14}\le u_{13}+u_{34}$, set $T^*=(M_{12},M_{23},M_{34},M_{13})$
and ${\bf u}^*=(u_{12},u_{23},u_{34},u_{13})$. If $u_{14}> u_{13}+u_{34}$, set $T^*=(M_{12},M_{23},M_{34},M_{13}, M_{14})$
and ${\bf u}^*=(u_{12},u_{23},u_{34},u_{13},u_{14})$.
\item Assume that $u_{13}\le u_{12}+u_{23}$ and $u_{24}> u_{23}+u_{34}$. If $u_{14}\le u_{12}+u_{24}$, set $T^*=(M_{12},M_{23},M_{34},M_{24})$
and ${\bf u}^*=(u_{12},u_{23},u_{34},u_{24})$. If $u_{14}> u_{12}+u_{24}$, set $T^*=(M_{12},M_{23},M_{34},M_{24}, M_{14})$
and ${\bf u}^*=(u_{12},u_{23},u_{34},u_{24},u_{14})$.
\item Assume that $u_{13}> u_{12}+u_{23}$ and $u_{24}> u_{23}+u_{34}$. 

If $u_{14}\le \max\{u_{12}+u_{24},u_{13}+u_{34}\}$, set $T^*=(M_{12},M_{23},M_{34},M_{13},M_{24})$
and ${\bf u}^*=(u_{12},u_{23},u_{34},u_{13},u_{24})$.

If $u_{14}> \max\{u_{12}+u_{24},u_{13}+u_{34}\}$,
set $T^*=T$, ${\bf u}^*={\bf u}$.\end{enumerate}
The tuple $T^*,{\bf u}^* $ and $\boldsymbol{\beta}^*=1/{\bf u}^*$ give three measures, the radial measure $\rho_{T^*,{\bf u}^*}$, the coordinate-wise stable measure $\nu_{T^*,{\boldsymbol \beta}^*}$, and the $1$d-singular measure $\mu_{T^*,{\boldsymbol \beta}^*}$. These three measures have equivalent Dirichlet forms which are also equivalent to the Dirichlet form of our original measure $\nu$. The probability of return of $\nu$, $\nu^{(2n)}(e)\asymp n^{-\gamma}$, has $\gamma$ equal to the sum of the entries of the tuple ${\bf u^*}$ given in each of the cases above.  

As a concrete example, 
consider the following convex combination of coordinate-wise and 1d-singular probability measures
$$\nu = \frac{1}{2}\nu_{S_1,\bm{\alpha}_1} + \frac{1}{2}\rho_{S_2,\bm{\alpha}_2} $$
where 
\begin{align*}
    S_1 &= (M_{12}, M_{23}, M_{13}, M_{14}),\quad \bm{\alpha}_1 = (1/3, 1/3, 1/4, 1/5)\\
     S_2 &= (M_{12}, M_{34}, M_{24}),\quad \bm{\alpha}_2  = (1/3, 1/3, 1/5)
\end{align*}
Following the notations above, we extract
$$T = (M_{12}, M_{23}, M_{34}, M_{13}, M_{24}, M_{14}) \text { and }\bm{\beta} = (1/3, 1/3,1/3, 1/4, 1/5, 1/5) $$
and it follows that
$$T^* = (M_{12}, M_{23}, M_{34}) \text{ and }\bm{\beta}^* = (1/3, 1/3,1/3). $$
Theorem \ref{thm: main1} states that the Dirichlet form $\mathcal E_\nu$ is equivalent to the Dirichlet form $\mathcal E_{\nu_{T^*,\bm{\beta}^*}}$.
\end{exa}

\subsection{Comparison of Dirichlet forms: Proof of Theorem \ref{thm: main1}}
\subsubsection{Radial and $1d$ singular measures}
In this section we consider a $1d$ singular measure $\mu_{S,\bm{\alpha}}$, $S=(s_1,\dots, s_k)$, $\bm{\alpha}=(\alpha_1,\dots,\alpha_k)$, and
the associated radial measure $\rho_{S,\mathbf v}$ where $\mathbf v=(v_1,\dots,v_k)$, with $v_i=1/\alpha_i>1/2$.

\begin{theo} 
\label{thm: rho_mu}
The following two Dirichlet forms are comparable
 $$\mathcal{E}_{\rho_{S,\mathbf v}} \asymp \mathcal{E}_{\mu_{S, \bm{\alpha}}}.$$
\end{theo}
The proof of this result proceeds by establishing two auxiliary theorems that are themselves of independent interest. The first of these two theorem is taken  from \cite{SCZ-nil}.
\begin{theo}[\cite{SCZ-nil,CKSCWZ1}] \label{thm: pp-inequality} For any $n \in N$,  the following pseudo-Poincar\'e inequality hold 
$$  \sum_{m\in N}|f(mn)-f(m)|^2\le C \|n\|_{S,\mathbf{v}} \mathcal E_{\mu_{S, \bm{\alpha}}}(f,f).$$
\end{theo} 
\begin{theo}
\label{theo: commu-dirichlet}
Set $s = [s_{1}, [s_{2} \ldots, [s_{k-1}, s_k]]]$ and  $ v =  \sum_{i=1}^k v_i
$. 
Then $$\mathcal{E}_{\mu_{s,\frac{1}{v}}} \preccurlyeq  \mathcal{E}_{\mu_{S,\bm{\alpha}}} $$
where 
$\mu_{s, \frac{1}{v}}$ is the probability measure with
$$\mu_{s,\frac{1}{v}}(g) \asymp \sum_{a\in \mathbb{Z}} \frac{\bm{1}_{s^a}(g)}{(1+|a|)^{1+\frac{1}{v}}}, \quad g\in N $$
\end{theo}

\begin{proof}
Suppose $N$ is of nilpotency class $c$. We'll prove the claim by backward induction on the length $k$. Assume $k = c$. For any $z \in \mathbb{Z}$ (we treat the case $z>0$ for definiteness) 
    $$s^z = \left[s_{1}^{\lfloor z^{v_1/v }\rfloor }, \ldots, s_{k}^{\lfloor z^{ v_k/v}\rfloor}\right] s^{
    z - \prod_{i=1}^k \lfloor z^{ v_i /v}\rfloor}$$
For any $z\in \mathbb{Z}_{\ge 0}$, denote by $r_z =  z - \prod_{i=1}^k \lfloor z^{ v_i /v}\rfloor$. By Lemma \ref{lem: decomp}, 
\begin{align}
    \mathcal{E}_{\mu_{s,\frac{1}{v}}} (f,f) & =  \sum_{z\in \mathbb{Z}} \sum_{g \in N} \frac{|f(g s^z) - f(g)|^2}{(1 + |z|)^{1+ 1/v}} \nonumber\\
     & \preccurlyeq \sum_{i=1}^k \sum_{z\in \mathbb{Z}}\sum_{g \in N}  \frac{\left|f\left(g s_{i}^{\lfloor z^{ v_i /v}\rfloor}\right) - f(g)\right|^2}{(1 + |z|)^{1+ 1/v}} + \sum_{z\in \mathbb{Z}}\sum_{g \in N}  \frac{\left|f\left(gs^{r_z}\right) - f(g)\right|^2}{(1 + |z|)^{1+ 1/v}} \label{eq: bracket_decomp}
\end{align}
For each $i=1,\ldots, k$, let's focus on the corresponding term in the first summand in Equation \ref{eq: bracket_decomp}:
\begin{align*}
    \sum_{z\in \mathbb{Z}}\sum_{g \in N}  \frac{\left|f\left(gs_{i}^{\lfloor z^{ v_i /v}\rfloor}\right) - f(g)\right|^2}{(1 + |z|)^{1+ 1/v}} & \preccurlyeq  \sum_{x \in \mathbb{Z}} \sum_{g \in N} \frac{|f(g s_i^x) - f(g)|^2(1+|x|)^{
    v/v_i
    -1}}{(1 + |x|^{ 
    v/v_i})^{1+1/v}} \\ 
    & \preccurlyeq \sum_{x \in \mathbb{Z}} \sum_{g \in N} \frac{|f(g s_i^x) - f(g)|^2}{(1 + |x|)^{1+\frac{1}{v_i}}} \preccurlyeq \mathcal{E}_{\mu_{s_i,\alpha_i}} (f,f)  
\end{align*}
In the first inequality, we express the sum in terms of a new parameter $x$ with $|x|\asymp |z|^{v_i /v}$ and note that at most $|x|^{v/v_i-1}$ integers $z$ give the same $x$. For the
second summand in Equation \ref{eq: bracket_decomp}, by results in Remark \ref{rem: norm}  and 
Lemma \ref{lem: floor}, 
$$\|s^{r_z}\|_{S,\mathbf{v}} \preccurlyeq (r_z)^\frac{1}{v} \preccurlyeq (1+|z|)^{\frac{1}{v} - \frac{v_*}{v^2}}$$
    where $v_* = \min \{v_i: i=1,\ldots,k \}$. 
Then Theorem \ref{thm: pp-inequality} implies 
    \begin{align*}
    \sum_{z\in \mathbb{Z}}\sum_{g \in N}  \frac{\left|f\left(gs^{r_z}\right) - f(g)\right|^2}{(1 + |z|)^{1+ 1/v}} \preccurlyeq   \sum_{z\in \mathbb{Z}}\frac{    \|s^{r_z}\|_{S ,\mathbf{v} } \mathcal{E}_{\mu_{S,\bm{\alpha}}} (f,f)   }{(1 + |z|)^{1+ 1/v}}
   \preccurlyeq   \sum_{z\in \mathbb{Z}}\frac{\mathcal{E}_{\mu_{S,\bm{\alpha}}} (f,f) }{(1 + |z|)^{1+ \frac{v_*}{v^2 }}} 
    \end{align*}
which is bounded by some constant multiple of $\mathcal{E}_{\mu_{S,\bm{\alpha}}} (f,f) $ as the sum converges.  This proves the case $k=c$.
Now assume $k < c$. For any $z\in \mathbb{Z}$ (assume $z>0$ again for definiteness),  we can  write
$$s^z = \left[(s_{1})^{\lfloor z^{ v_1/v }\rfloor }, \ldots, (s_{k})^{\lfloor z^{ v_k/v}\rfloor}\right] s^{
    z - \prod_{i=1}^k \lfloor z^{ v_i/v}\rfloor} \cdot \xi $$
where $\xi$ is a product of power of commutators of length greater than $k$. A similar argument to that used in the base case and the induction hypothesis, combined with the  inequality established by Lemma \ref{lem: decomp}, yields the desired result.

\end{proof}

\begin{proof} (of Theorem \ref{thm: rho_mu})
Throughout the proof, we repeatedly invoke the notation introduced in Section \ref{sec: series} and
may drop the subscripts in  $\rho_{S,\mathbf{v}}$,
$ \|\cdot\|_{S,\mathbf{v}}$ and $\gamma_{S,\mathbf{v}}$. In \cite{SCZ-nil}, it's proved that $ \mathcal{E}_{\mu_{S,\bm{\alpha}}}\preccurlyeq \mathcal{E}_{\rho_{S,\mathbf{v}}}$. It remains to show the other direction of inequality.  For $i=1,\ldots, j_*$, relabel the commutators in $\mathcal{C}_{w_i}(S)$ as $ \{c_{i,r}, r  = 1,\ldots, \delta_i\}$ where $\delta_i = |\mathcal{C}_{w_i}(S)|$. Write 
   $$\mathcal{E}_{\rho_{S,\mathbf{v}}}(f,f)  = \sum_{k=1}^\infty \sum_{\substack{h \in N: \\
    2^{k-1} \le \|h\| <  2^{k}}} \frac{\sum_{g \in N} |f(gh) - f(g)|^2}{ \|h\|^{1 +\gamma}}$$
    For $R\ge 1$, let $\chi(R)$ be a collection of tuples defined by
    $$\chi(R) = \left\{(x_{i,r})_{\subalign{ i &= 1, \ldots, j_* \\  r &= 1,\ldots, \delta_i}}: |x_{i,r}| \le C\max\left(|c_{i,r}|, R^{w_i} \right) \right\} $$
where $|g|$ denotes the group order of the group element $g$. 
Applying the results in \cite{SCZ-nil}, we derive  the following upper bound for $ $ 
\begin{align*}
    \mathcal{E}_{\rho_{S,\mathbf{v}}}(f,f)  & \preccurlyeq  \sum_{k=1}^\infty 
    \sum_{ x \in \chi(2^k)
 } \frac{\sum_{g} \left |f\left(g  \prod_{i=1}^{j_*}\prod_{r = 1}^{\delta_i} c_{i,r}^{x_{i,r}}
\right) - f(g)\right|^2}{ 2^{k(1 + \gamma)}} \\
& \preccurlyeq \sum_{i=1}^{j_*}\sum_{r = 1}^{\delta_i}  \sum_{k=1}^\infty  \sum_{ x \in \chi(2^k)
 }  \frac{\sum_{g} |f(g c_{i,r}^{x_{i,r}} ) - f(g)|^2}{ 2^{k(1 + \gamma)} } = \sum_{i=1}^{j_*}\sum_{r = 1}^{\delta_i}  A_{i,r} 
\end{align*}
where for every $(i,r )$ pair
$$A_{i,r} := \sum_{k=1}^\infty  \sum_{ x \in \chi(2^k)
 }  \frac{\sum_{g} |f(g c_{i,r}^{x_{i,r}} ) - f(g)|^2}{ 2^{k(1 + \gamma)} } $$
 Clearly, if $c_{i,r}$ is torsion in $N$, $A_{i,r} \preccurlyeq \mathcal{E}_{\mu_{S, \bm{\alpha}}}(f,f)$ trivially. 
 For a non-torsion $c_{i,r}$, the corresponding  $A_{i,r}$ can be rewritten as follows: 
\begin{align*}
A_{i,r} & \asymp   \sum_{k=1}^\infty  \sum_{ |x_{i,r}| \le C 2^{k w_i} }  \frac{\sum_{g} |f(g c_{i,r}^{x_{i,r}} ) - f(g)|^2}{ 2^{k(1 + w_i)} } \\
    & = 2\sum_{x=1}^\infty \sum_{g}  |f(g  c_{i,r}^x) - f(g)|^2 \sum_{k \ge \frac{\log_2(x)}{ w_i}} \frac{1}{2^{k(1 + w_i)}}\\
    &\asymp 2 \sum_{x=1}^\infty  \frac{\sum_{g } |f(g c_{i,r}^{x}) - f(g)|^2 }{|x|^{1 + 1/w_i}} = 2 \mathcal{E}_{c_{i,r}, \frac{1}{w_i}} (f,f) \preccurlyeq \mathcal{E}_{\mu_{S,\bm{\alpha}}}(f,f) 
\end{align*}
The last inequality follows from the fact that, 
 by construction, $w_i$ is exactly the sum of the assigned weight (i.e. the $v_\bullet$'s) of the generators $s_\bullet$'s appearing in the commutator $c_{i,r}$ and Theorem \ref{theo: commu-dirichlet}. This completes the proof. 
\end{proof}

\subsubsection{Coordinate-wise and $1d$-singular measures}
In this section we consider a $1d$ singular measure $\mu_{S,\bm{\alpha}}$ and the associated coordinate-wise measure $\nu_{S,\bm{\alpha}}$ where 
 $S=(s_1,\dots, s_k)$ and $\bm{\alpha}=(\alpha_1,\dots,\alpha_k) \in (0,2)^k$. 

\begin{theo}
\label{thm: coord}
  The following two Dirichlet forms are comparable  $$\mathcal{E}_{\mu_{S,\bm{\alpha}}} \asymp  \mathcal{E}_{\nu_{S, \bm{\alpha}}}$$
\end{theo}
\begin{proof} In this proof,  for a $k$-dimensional vector $\bar{a}=(a_1,\dots,a_k)$, $|\bar{a}|= \sum_{i=1}^k |\bar{a}_i|$ denotes the 1-norm.

Note that 
$\mathcal{E}_{\mu_{S,\bm{\alpha}}}  (f,f)  \asymp \sum_{i=1}^k \mathcal{E}_{\mu_{\{s_i\},\alpha_i}}(f,f) $. By Lemma \ref{lem: decomp} and Lemma \ref{lem: sum}, 
$ \mathcal{E}_{\nu_{S, \bm{\alpha}} }(f,f)$ is bounded above ($\preccurlyeq$) by
    \begin{align*}
          \sum_{\bar{a}\in \mathbb{Z}^k} \frac{ \sum_{g\in N}|f(g \pi(\bar{a})) - f(g)|^2 }{ \left(1+ \sum_{i=1}^k |\bar{a}_i|^{\alpha_i}\right)^{1+ \sum_{i=1}^k \frac{1}{\alpha_i}} }
         \preccurlyeq   \sum_{j=1}^k  \sum_{\bar{a}\in \mathbb{Z}^k}
        \frac{ \sum_{g\in N} |f(g s_j^{\bar{a}_j}) - f(g)|^2 }{ \left(1+ \sum_{i=1}^k |\bar{a}_i|^{\alpha_i}\right)^{1 + \sum_{i=1}^k \frac{1}{\alpha_i}} } \\
          = \sum_{j=1}^k  \sum_{
      \bar{a}\in \mathbb{Z}^k
     }
        \frac{  \sum_{g\in N} |f(g s_j^{\bar{a}_j}) - f(g)|^2 }{(1 + |\bar{a}_j|^{\alpha_j} + \sum_{i\ne j}|\bar{a}_i|^{\alpha_i}
        )^{1 + \sum_{i=1}^k \frac{1}{\alpha_i}}}
        \asymp \sum_{j=1}^k  \sum_{
       a \in \mathbb{Z} }
        \frac{ \sum_{g\in N} |f(g s_j^{a}) - f(g)|^2 }{(1 + |a|^{\alpha_j})^{1 + \frac{1}{\alpha_j}}}  \\
         \asymp  \sum_{j=1}^k  \sum_{
       a \in \mathbb{Z} }
        \frac{ \sum_{g\in N} |f(g s_j^{a}) - f(g)|^2 }{(1 + |a|)^{\alpha_j + 1}}  
         = \sum_{j=1}^k 
 \mathcal{E}_{\mu_{\{s_j\},\alpha_j}}(f,f) 
    \end{align*}
This proves 
 $\mathcal{E}_{\nu_{S, \bm{\alpha}}}(f,f) \preccurlyeq \mathcal{E}_{\mu_{S, \bm{\alpha}}}(f,f)$.

To obtain the reverse inequality, we will prove that  $\mathcal{E}_{\mu_{\{s_i\},\alpha_i}} (f,f)  \le \mathcal{E}_{\nu_{S, \alpha}}(f,f)$ by induction on $i$. Consider the base case $\mathcal{E}_{\mu_{\{s_1\},\alpha_1}}(f,f)  $. Note that for any 
 $a \in \mathbb{Z}^1$ and $\bar{v} \in 0 \times \mathbb{Z}^{k-1}$, i.e. a  $k$-dimensional integer-valued vector with the first entry 0, we can write
$s_1^{a} = s_1^{2a}  \pi(\bar{v}) \cdot \pi(\bar{v})^{-1} s_1^{-a}$. Again, by Lemma \ref{lem: decomp} and Lemma \ref{lem: sum}, 
\begin{align*}
&\mathcal{E}_{\mu_{\{s_1\},\alpha_1}} (f,f)  
 =  \sum_{a \in \mathbb{Z}}\sum_{g\in G} |f(gs_1^{a}) - f(g)|^2  \frac{1}{(1+ |a|)^{1+\alpha_1}} \\
    &\asymp \sum_{a \in \mathbb{Z}}\sum_{g\in G} |f(gs_1^{a}) - f(g)|^2  \sum_{
\bar{v}  \in 0\times \mathbb{Z}^{k-1}} \frac{1}{(1 + |a|^{\alpha_1} + \sum_{i=2}^k 
|\bar{v}_i|^{\alpha_i})^{1 + \sum_{i=1}^k \frac{1}{\alpha_i}}} \\
   & \preccurlyeq 
    \sum_{ 
\substack{ a \in \mathbb{Z}\\\bar{v} \in 0\times \mathbb{Z}^{k-1}}}  
   \frac{\sum_{g\in G} |f(g) - f(gs_1^{2a}  \pi(\bar{v} ))|^2 + \sum_{g\in G} |f(gs_1^{a}  \pi(\bar{v} )) - f(g)|^2  }{(1 + |a|^{\alpha_1} + \sum_{i=2}^k 
|\bar{v}_i|^{\alpha_i})^{1 + \sum_{i=1}^k \frac{1}{\alpha_i}}} \\
   & \preccurlyeq \mathcal{E}_{\nu_{S,\bm{\alpha}}}(f,f)
\end{align*}

Now consider $1 < p \le k$. For any $a\in \mathbb{Z}$, 
$\bar{v} \in \mathbb{Z}^{p-1} \times 0^{k-p+1} =: V_1$ and $\bar{w} \in  0^{p} \times \mathbb{Z}^{k-p}=: V_2$, we can decompose 
\begin{align*}
    s_p^{a} = 
   \prod_{i=p-1}^{1} s_i^{-\bar{v}_i}\cdot 
    \pi(\bar{v}) s_p^{2a}   \pi(\bar{w}) \cdot  \pi(\bar{w})^{-1}s_p^{-a}  \pi(\bar{v})^{-1}\cdot  \prod_{i=1}^{p-1} s_i^{\bar{v}_i} 
\end{align*}
It again follows from Lemma \ref{lem: decomp} and Lemma \ref{lem: sum} that, 
\begin{align*}
    & \mathcal{E}_{\mu_{\{s_p\},\alpha_p}}(f,f)  \\
    & = \sum_{a\in \mathbb{Z}} \sum_{g\in G}   \frac{|f(gs_p^{a}) - f(g)|^2}{(1+ |a|)^{1+\alpha_p}} \\
 & \asymp  \sum_{\substack{ a \in \mathbb{Z}\\\bar{v} \in V_1\\ \bar{w} \in V_2} } \sum_{g\in G}    
\frac{|f(gs_p^{a}) - f(g)|^2 }{ (1+ |a|^{\alpha_p} + 
\sum_{i=1}^{p-1} |\bar{v}_i|^{\alpha_i} + \sum_{i=p+1}^k |\bar{w}|^{\alpha_k})^{1 + \sum_{i=1}^k \frac{1}{\alpha_i}}}\\
&\preccurlyeq    \sum_{\substack{ a \in \mathbb{Z}\\\bar{v} \in V_1\\ \bar{w} \in V_2} }  \sum_{g\in G}    \frac{ | f(g  \pi(\bar{v}) s_p^{2a}   \pi(\bar{w}))- f(g )|^2 + |f(g \pi(\bar{v}) s_p^{a}   \pi(\bar{w})) - f(g) |^2}{ (1+ |a|^{\alpha_p} + 
\sum_{i=1}^{p-1} |\bar{v}_i|^{\alpha_i} + \sum_{i=p+1}^k |\bar{w}|^{\alpha_k})^{1 + \sum_{i=1}^k \frac{1}{\alpha_i}}}  \\
 &\quad +  2 \sum_{j=1}^{p-1} \sum_{\substack{ a \in \mathbb{Z}\\\bar{v} \in V_1\\ \bar{w} \in V_2} }  \sum_{g\in G}    \frac{  |f(g) - f(gs_j^{\bar{v}_j})|^2 }{ (1+ |a|^{\alpha_p} + 
\sum_{i=1}^{p-1} |\bar{v}_i|^{\alpha_i} + \sum_{i=p+1}^k |\bar{w}|^{\alpha_k})^{1 + \sum_{i=1}^k \frac{1}{\alpha_i}}} \\
 & 
 \preccurlyeq 
2 \mathcal{E}_{\nu_{S,\bm{\alpha}}}(f,f) + 
  2 \sum_{j=1}^{p-1} \sum_{\bar{u}\in \mathbb{Z}^k }  \sum_{g\in G}    \frac{  |f(g) - f(gs_j^{\bar{u}_j})|^2 }{(1+  
  \sum_{i=1}^k
  |\bar{u}_i|^{\alpha_i})^{1 + \sum_{i=1}^k \frac{1}{\alpha_i}}}\\
  & \preccurlyeq 2\mathcal{E}_{\nu_{S,\bm{\alpha}}}(f,f) + 2\sum_{j=1}^{p-1} \mathcal{E}_{\mu_{\{s_j\},\alpha_j}}(f,f) 
     \end{align*}
     Finally, each of $\mathcal{E}_{\mu_{\{s_j\},\alpha_j}}(f,f) $ for $i=1,\ldots, p-1$ has upper bound $ \mathcal{E}_{\nu_{S,\bm{\alpha}}}(f,f) $ by induction hypothesis, completing the proof. 
\end{proof}

\section{On a group of polynomial growth $
G$}
Recall that $G$ is a group of polynomial volume growth and that $N$ is a normal nilpotent subgroup of $G$ with finite index. The set
$\{x_0=e,x_1,\dots,x_p\}$ is a set of coset representatives for $N$ in $G$ (left or right, it does not matter, as $N$ is normal).

We begin this section with two rather general theorems providing condition under which one can  compare Dirichlet forms associated with measures on $G$ and $N$. Throughout this section, let $N$ be normal in $G$ with finite coset representative $X$.
\begin{lem}
\label{lem: Dirichlet-decomp} 
Let  $\nu$ be a symmetric probability measure on $N$ with the property that $\nu(n) \asymp \nu(xnx^{-1})$ for all $x\in X$ and $n\in N$.  
Then, for any finitely supported function $f$ on $G$ and its restriction $\left. f\right|_N$ to $N$, we have  
$$\mathcal{E}_{\nu, G}(f,f)\preccurlyeq  \mathcal{E}_{\nu, N}(f|_N,f|_N) + \mathcal{E}_{\kappa_X, G}(f,f)$$
where $\kappa_X$ is the uniform measure on $X\cup X^{-1}$. 
\end{lem}

\begin{proof}
With the identification $G = NX$, we have
\begin{align*}
     & \mathcal{E}_{\nu, G} (f,f) 
     = \sum_{\substack{x\in X\\ m\in N\\n\in N}} |f(mxn) - f(mx)|^2 \nu(n)\\
      & = \sum_{\substack{x\in X\\ m\in N\\n\in N}} |f(mxnx^{-1}x) - f(mx)|^2 \nu(x^{-1}nx)  \asymp \sum_{\substack{x\in X\\ m\in N\\n\in N}} |f(mnx) - f(mx)|^2 \nu(n) \\
\end{align*}
Now, for every $x \in X, m \in N, n\in N$, write 
$$|f(mnx) - f(mx)|^2 \preccurlyeq |f(mnx) -f(mn)|^2 + |f(mn) - f(m)|^2 + |f(m) - f(mx)|^2$$
and it follows 
\begin{align*}
    &  \sum_{\substack{x\in X\\ m\in N\\n\in N}} |f(mnx) - f(mx)|^2 \nu(n)  \\
    \preccurlyeq &   \sum_{\substack{x\in X\\ m\in N\\n\in N}} |f(mnx) - f(mn)|^2 \nu(n) +  \sum_{\substack{x\in X\\ m\in N\\n\in N}} |f(mn) - f(m)|^2 \nu(n) \\
     & + 
    \sum_{\substack{x\in X\\ m\in N\\n\in N}} |f(m) - f(mx)|^2 \nu(n)  \\
   \asymp  &   \sum_{\substack{x\in X\\ m\in N}} |f(mx) - f(m)|^2 \sum_{n \in N} \nu(n) + \mathcal{E}_{\nu, N}(f|_N, f|_N) 
   \\
    & \asymp \mathcal{E}_{\kappa_X, G}(f,f) + \mathcal{E}_{\nu, N}(f|_N, f|_N)  
\end{align*}
as desired. 
\end{proof}

\begin{theo}
\label{theo: GN-Dirich}
    Let $\nu_G$ be a symmetric probability measure with support $\Gamma \subseteq G$. 
    Suppose there exists a finite set $Y\subset \Gamma$ satisfying $\Gamma = (\Gamma \cap N) Y$ and for any $\gamma = ny$, $\gamma \in \Gamma, n\in \Gamma \cap N$ and $y \in Y$, 
    $\mu_G(\gamma) \asymp \mu_G(n)$. Let 
    $\nu_N$ be a symmetric probability measure supported on 
    $$\bigcup_{x \in X} x (\Gamma\cap N) x^{-1}$$ with 
    $$\nu_N (n) \asymp \nu_G(n) \quad \forall n \in \Gamma \cap N$$
    and 
    $$\nu_N(n) \asymp \nu_N(xnx^{-1}) \quad \forall n \in \Gamma \cap N, x \in X$$
Then     
$$ \mathcal{E}_{\nu_G}(f,f) + \mathcal{E}_{\kappa_X, G}(f,f) + \mathcal{E}_{\kappa_Y}(f,f) \asymp \mathcal{E}_{\kappa_X} (f,f) + \mathcal{E}_{\kappa_Y}(f,f)   + \mathcal{E}_{\nu_N, N}(f|_N,f|_N) $$
\end{theo}
\begin{proof} First we bound $\mathcal{E}_{\nu_G}(f,f)$ from above.
    \begin{align*}
        \mathcal{E}_{\nu_G}(f,f) & = \sum_{g \in G, h\in \Gamma} |f(gh) - f(g)|^2 \nu_G(h)  = \sum_{\substack{g \in G\\n \in \Gamma \cap N\\ y\in Y}}  |f(gny) - f(g)|^2 \nu_N(n)\\
        &\preccurlyeq \sum_{\substack{g \in G\\n \in \Gamma \cap N\\ y\in Y}}  |f(gny) - f(gn)|^2 \nu_N(n) + \sum_{\substack{g \in G\\n \in \Gamma \cap N\\ y\in Y}}  |f(gn) - f(g)|^2 \nu_N(n)\\
        & = \sum_{\substack{g \in G\\ y\in Y}}  |f(gy) - f(g)|^2 \sum_{n \in \Gamma \cap N}\nu_N(n) + |Y|\mathcal{E}_{\nu_N, G}(f,f)\\
        &\asymp \mathcal{E}_{\kappa_Y, G}(f,f) + \mathcal{E}_{\nu_N, G}(f,f)\\
        & \preccurlyeq \mathcal{E}_{\kappa_X} (f,f) + \mathcal{E}_{\kappa_Y}(f,f)   + \mathcal{E}_{\nu_N, N}(f|_N,f|_N)  \\
    \end{align*}
To prove the other direction, it suffices to prove $\mathcal{E}_{\nu_N, N}(f|_N,f|_N) \preccurlyeq \mathcal{E}_{\nu_G} (f,f) $.  
Note that
\begin{align*}
    \mathcal{E}_{\nu_{N}, N}(f|_N,f|_N) & \asymp \sum_{\substack{n\in N\\ x \in X \\ \gamma \in \Gamma \cap N\\}} |f(n x \gamma x^{-1}) - f(n)|^2\nu_N(\gamma)
\end{align*}
Obviously, we have
\begin{align*}
   & |f(n x\gamma  x^{-1}) - f(n)|^2 \\
     \preccurlyeq & |f(n x\gamma  x^{-1}) 
 -f(n x\gamma  ) |^2 + |f(n x\gamma  ) - f(nx)|^2 + |f(nx) - f(n)|^2
\end{align*}
The first term yields
\begin{align*}
    \sum_{\substack{n\in N\\ x \in X \\ \gamma \in \Gamma \cap N}}  |f(n x\gamma  x^{-1}) 
 -f(n x\gamma  ) |^2 \nu_N(\gamma) 
& \preccurlyeq\sum_{\substack{\xi\in N\\ x \in X }}  
|f(\xi) 
 -f(\xi x) |^2 \sum_{\gamma \in \Gamma \cap N}\nu_N(\gamma ) \\
 & \preccurlyeq \mathcal{E}_{\kappa_X}(f,f) .
\end{align*}
For the second we get
$$  \sum_{\substack{n\in N\\ x \in X \\ \gamma \in \Gamma \cap N}}  
 |f(n x\gamma  ) - f(nx)|^2  \nu_N(\gamma)
 \preccurlyeq \sum_{\substack{g\in G\\ \gamma \in \Gamma \cap N}}  
 |f(g\gamma  ) - f(g)|^2  \nu_G(\gamma) \preccurlyeq\mathcal{E}_{\nu_G}(f,f)$$
and contribution of the third is obviously bounded by $\mathcal{E}_{\kappa_X}(f,f)$. This completes the proof. 
\end{proof} 

\subsection{Relating measures on $G$ to measures on $N$} 
\label{sec: G_to_N}
Fix a tuple $S=(s_1\dots,s_k) \in G^k$ and a weight tuple $\mathbf{v}=(v_1,\dots,v_k)\in (1/2,\infty)^k$.

The only elements $s \in S$ itself for which the associated $v_i$ can play a critical role  are those that are not torsion, i.e., those for which the cyclic group  $\langle s\rangle $ is infinite. By the second isomorphism theorem, this is the case if and only if $\langle s\rangle \cap N$ is infinite because
$\langle s\rangle N/N\sim \langle s\rangle /\langle s\rangle \cap N$.
For such $s\in S$ there exist a unique positive integer $\epsilon$ such that $\langle s\rangle \cap N=\langle s^\epsilon\rangle$. For each $i\in \{1,\dots,k\} $ such that $\langle s\rangle$ is infinite, let $\epsilon_i$ be the positive integer such that $\langle s_i\rangle \cap N=\langle s_i^{\epsilon_i}\rangle$. Define the following tuples:
\begin{itemize}
\item the tuple $\widetilde{S_G}$ of elements in $N\subset G$ which list, in order,
the elements $s_i^{\epsilon_i}\in N$ when $i$ runs over the indices for which $\langle s_i\rangle \cap N$ is infinite.
\item The tuple $\widetilde{S_N}$ of elements in $N$ obtained by concatenating  
$$\widetilde{S_G}, x_1\widetilde{S_G} x_1^{-1},\dots,x_p\widetilde{S_G}x_p^{-1},$$ i.e.,
$$\widetilde{S_N} :=  \bigsqcup_{j=0}^p \{x_j \widetilde{S_G} x_j^{-1}\} .$$ 
\end{itemize}

Observe that the first tuple, $\widetilde{S_G}$ does not necessarily generate $G$ (nor $N$) and that the second tuple, $\widetilde{S_N}$ also does not necessarily generate $N$. For this reason, we introduce an arbitrary finite generating tuple $S_0$ for $G$ (for instance, $S$), and an arbitrary
finite generating tuple $\Sigma_0$ for $N$.

\begin{defin}
\label{def: norm_G}
Let $\widetilde{S_G}$ be the formal concatenation of $S_0$ and $S_G$.
Define the weight system $w_G$ on $\widetilde{S_G}$ by setting 
$$
w_G(\sigma) = \begin{cases}
\frac{1}{2} &\text{if }\sigma \in S_0\\
v_i &\text{if }\sigma = s_i^{\epsilon_i} \in S_G.\\
\end{cases}
$$
Set 
$$ \|g\|_G =   \|g\|_{G,S,\mathbf{v}} = \inf \left\{ \max_{\sigma \in \widetilde{S_G} }\left\{ \mbox{deg}_\sigma(\theta)^\frac{1}{w_G(\sigma) } \right\}:  \theta = g \in G, \theta \in \bigcup_{m=0}^\infty \widetilde{S_G}^m\right\} $$
\end{defin}
Note that we can choose $S_0=S$ since $S_0$ is a finite generating tuple for $G$. The particular choices made  of $S_0$ and $\Sigma_0$ are irrelevant in what follows.

\begin{defin}
\label{def: norm_N}
    Let $\widetilde{S_N}$ be the formal concatenation of $\Sigma_0$ and $S_N$. 
Define the weight system $w_N$  by setting 
$$w_N(\sigma) = 
 \begin{cases}
\frac{1}{2} &\text{if }\sigma \in \Sigma_0\\
v_i &\text{if }\sigma =x_j s_i^{\epsilon_i}  x_j^{-1}  \subseteq S_N.\\
\end{cases}
$$  Set 
$$    \|g\|_N =  \|g\|_{N,S,\mathbf{v}} = \inf \left\{ \max_{\sigma \in \widetilde{S_N}}\left\{ \mbox{deg}_\sigma(\theta)^\frac{1}{w_N(\sigma) } \right\}:  \theta = g \in G, \theta \in \bigcup_{m=0}^\infty \widetilde{S_N}^m\right\}\\ $$
Regard $S_N$ as a tuple and denote by $\mathbf{v}_N$ the tuple given by $(w_N(\sigma))_{\sigma \in S_N}$. 
\end{defin}

These definitions have their origin in \cite{SCZ-nil,CKSCWZ1}. The reader should note that these definitions refers to tuples in $N$ and $G$. That is, in these definition, $\sigma$ is regarded as an element of the tuples $\widetilde{S_G}$ or $\widetilde{S_N}$, not just an element of the corresponding group. In fact, we will often think of the elements of these tuples as letter in an abstract alphabet. This means that it is possible that one element $\sigma$ of $N$ (or $G$)
appears multiple times as distinct letters in the tuple $\widetilde{S_N}$ (or $\widetilde{S_G}$). When that's the case, the different occurrences of $\sigma$ may receive different weights.
This interpretation is also important in the definition of the (quasi-)norm with $\|g\|_G$ and $\|g\|_N$.

\begin{theo}[{\cite[Theorem 3.2.1]{CKSCWZ}}, {\cite[Th. 3.2 and Rm. 3.3]{SCZ-nil}}]
\label{thm: gamma_G}
 For all $g\in N$, $\|g\|_N\asymp \|g\|_G$. Furthermore, let $\gamma := \gamma(N, S_N, \mathbf{v}_N )$ be defined as in Equation  \ref{eq: gamma}, for all $R \ge 1$, 
    $$\left|\{g \in G: \|g\|_G\le R\}\right| \asymp \left|   \{g \in N: \|g\|_G\le R\}\right|\asymp \left|  \{g \in N: \|g\|_N\le R\}\right| \asymp R^{\gamma}$$
\end{theo}

One basic connection between the geometries induced by $\|\cdot\|_G$ and $\|\cdot\|_N$on $N$ and $G$ is captured by the following fundamental (and sharp) pseudo-Poincar\'e inequalities for the 1d  symmetric probability measures 
$$\mu_{S,\bm{\alpha}}(g) \asymp \sum_{i=1}^k \sum_{a\in \mathbb{Z}} \frac{\mathbf{1}_{ s_i^a}(g)}{(1+|a|)^{1+\alpha_i}}, \qquad \bm{\alpha} = 1/\mathbf{v}$$  
and 
\begin{equation}\label{muSigma}
\mu_{S_N, \bm{\alpha}_N} \asymp \sum_{i=1}^{|S_N|} \sum_{a\in \mathbb{Z}} \frac{\mathbf{1}_{ s_i^a}(g)}{(1+|a|)^{1+\tilde{\alpha}_i}}. \qquad \bm{\alpha}_N = 1/\mathbf{v}_N
\end{equation}
where $1/\mathbf{v}$ and 
$ 1/\mathbf{v}_N$ denote the component-wise reciprocal of $\mathbf{v}$ and  $\mathbf{v}_N$ respectively.  
The second inequality is merely a restatement of Theorem \ref{thm: pp-inequality}. 

\begin{theo}[\cite{SCZ-nil,CKSCWZ1}] 
For any $h \in G$, the following pseudo-Poincar\'e inequalities hold.\begin{itemize}
\item For any $h\in G$,   $$  \sum_{g\in G}|f(gh)-f(g)|^2\le C \|h\|_G \mathcal E_{\mu_{S,\bm{\alpha}}}(f,f).$$
\item For any $n \in \langle S_N\rangle$, 
$$  \sum_{m\in N}|f(mn)-f(m)|^2\le C \|n\|_N \mathcal E_{\mu_{S_N,\bm{\alpha}_N}}(f,f).$$
\end{itemize}
\end{theo}

A simple but instructive example is as follows. Although elementary, it highlights how the choice of $S$ and its interaction with $N$ affects the random walk behaviors, and motivates the construction above. 
\begin{exa}
Let $G = \mathbf{D} =  \langle u, v : u^2 = v^2\rangle$ be the infinite Dihedral group. It is not nilpotent because $[(uv)^n, u] = (uv)^{2n}$ and contains $N:= \langle uv \rangle \simeq \mathbb{Z}$ as the nilpotent subgroup with quotient $\langle u\rangle\simeq \mathbb{Z}_2$. 
Take $S = (u, v)$ and $\bm{\alpha} = (\alpha_1, \alpha_2) \in (0,2)^2$.  Because $u,v$ is torsion, $\mu_{S,\bm{\alpha}}$ is comparable to the uniform measure $\kappa_S$ on $S$ which satisfies
$$\mu_{S,\bm{\alpha}}^{(2n)}(e) \asymp \kappa_S^{(2n)}(e)\asymp 
n^{-\frac{1}{2}}$$
as $V_G(R) = \# \{g: |g| \le R\} = R^1$. 
This justifies 
our construction of $S_N$ and our choice of 
assigning the value $\frac{1}{2}$ to $w_G$ on $S_0$ (respectively, to $w_N$ on $\Sigma_0$), since under this convention the resulting set $S_N$ is empty and the associated measure exhibits the desired convolution behavior by Theorem \ref{thm: main2}. 

On the other hand, 
suppose $S' = (uv, u,v)$ and $\bm{\alpha}' = (\alpha_1,\alpha_2,\alpha_3) \in (0,2)^3$. In this case,  $S'_N = \{uv\}$ generates $N$ and the volume exponent given in Theorem \ref{thm: gamma_G} is $\gamma = \frac{1}{\alpha_1}$. So, by Theorem \ref{thm: main2}, 
$$\mu_{S',\bm{\alpha}'}^{(2n)}(e) \asymp  n^{-\frac{1}{\alpha_1}}$$
as expected. Note that this is the opposite extreme: the components that interact nontrivially with $N$ comprise the entire measure. In this case, the specific weight assignments on $S_0$ and $\Sigma_0$ are irrelevant.

For a more complex multidimensional analogue of the dihedral group, see Example 4.7 in \cite{CKSCWZ}.
\end{exa}

\subsection{Stable-like measures on $G$}

We consider a family of symmetric probability measures 
$v$ on $G$, defined in the same manner as those introduced in Section \ref{sec: convex}. Specifically, 
$$\nu=\sum _i^\ell c_i\nu_i$$
is a convex combination of probability measures, where each
$\nu_i$ is supported on a subgroup of $G$ and and the union of these supports generates $G$. Furthermore, for each $i\in \{1,\dots,\ell\}$, one of the following two possibility occurs:
\begin{enumerate}
    \item There is subgroup $H_i$ generated by a tuple $S_i=(s_{i,1},\dots,s_{i,k_i}) $ and there is a weight tuple $\mathbf v_i=(v_{i,1},\dots,v_{i,k_i})\in (1/2,+\infty)^{k_i}$ such that $\nu_i\asymp \rho_{G, S_i,\mathbf v_i}\asymp (1+\|\cdot\|_{G, S_i, \mathbf v_i})^{-1-\gamma_i}$ where 
    $\#\{\|g\|_{G, S_i,\mathbf v_i}\le r\}\asymp r^{\gamma_i}$.  In this case, we define the tuple $\bm{\alpha}_i=(\alpha_{i,1},\dots,\alpha_{i,k_i})$ by setting $\alpha_{i,j}=1/v_{i,j}$, $1\le j\le k_i$.
    \item There is a subgroup $H_i$ generated by a tuple $S_i=(s_{i,1},\dots,s_{i,k_i}) $ and there is an exponent tuple $\bm{\alpha}_i=(\alpha_{i,1},\dots,\alpha_{i,k_i})\in (0,2)^{k_i}$ such
    $\nu_i \asymp \nu_{S_i,\bm {\alpha}_i}$ of type (\ref{cwsl}). In this case, we also define a weight tuple $\mathbf v_i$ by setting $v_{i,j}=1/\alpha_{i,j}$, $1\le j\le k_i$.    \end{enumerate}

Again, we construct a tuple $S=S_\nu$ of group elements, a weight tuple $\mathbf v=\mathbf v_\nu$ and an exponent tuple $\bm{\alpha} = \bm{\alpha}_\nu$, all of the same length $k=\sum_1^\ell k_i$, by concatenation of the tuples $S_i$, $\mathbf v_i$, and $\bm{\alpha}_i$ associated to $\mu_i$ on $H_i$, $1\le i\le \ell$. Consider the associated radial measure on $G$ \begin{equation*}
\rho_{G, S,\mathbf v}\asymp(1+\|\cdot\|_{G, S,\mathbf v})^{-1-\gamma}
 \end{equation*}
where the parameter $\gamma$ is given by Theorem \ref{thm: gamma_G} and the associated 1d-singular probability measure $\mu_{S, \bm{\alpha}}$ on $G$ defined as in (\ref{muS}). 

\begin{theo}
\label{thm: main2}
    Given a symmetric stable-like probability measure $\nu$ on $G$, we have 
    $$\mathcal{E}_{\nu} \asymp \mathcal{E}_{\rho_{G, S,\mathbf v}} \asymp \mathcal{E}_{\mu_{S, \bm{\alpha}}}$$
    and all three measures satisfy comparable $n$-step convolution estimate, i.e. 
    $$\phi^{(2n)}(e) \asymp n^{-\gamma}, \qquad \phi \in \{\nu, \rho_{G, S,\mathbf v}, \mu_{S, \bm{\alpha}}\}$$
Furthermore, let $\tilde{\nu}$ be another stable-like probability measure in the collection, and form the tuple
$(\tilde{S}, \tilde{\mathbf{v}})$ by concatenating the tuples arising from its components, as before. Then 
$$\|\cdot\|_{G,S,\mathbf{v}}\asymp \|\cdot \|_{G,\tilde{S}, \tilde{\mathbf{v}}} \Longleftrightarrow \|\cdot\|_{N,S,\mathbf{v}}\asymp \|\cdot \|_{N,\tilde{S}, \tilde{\mathbf{v}}}  $$
Let $(S_N,\mathbf{v}_N)$ and $(\tilde{S}_N,\tilde{\mathbf{v}}_N)$ be the tuples associated with $(S,\mathbf{v})$ and $(\tilde{S},\tilde{\mathbf{v}})$, respectively, representing the components that interact nontrivially with the nilpotent subgroup $N$, as defined in Definition~\ref{def: norm_N}. The two conditions above are further equivalent to the following.
\begin{enumerate}
        \item $\mathfrak w_{S_N,\mathbf{v}_N}(s)=\mathfrak w_{\tilde{S}_N,\tilde{\mathbf{v}}_N}(s)$ for all $
s \in \mbox{\em core}(S_N,\mathbf{v}_N) \cup \mbox{\em core}(\tilde{S}_N,\tilde{\mathbf{v}}_N)$ 
    \item $\mathfrak w_{S_N,\mathbf{v}_N}(c)=\mathfrak w_{\tilde{S}_N,\tilde{\mathbf{v}}_N}(c)$ for all $
 c\in\mathcal{C}(\mbox{\em core}(S_N,\mathbf{v}_N)) \cup \mathcal{C}(\mbox{\em core}(\tilde{S}_N,\tilde{\mathbf{v}}_N))$ 
\item  $\mbox{\em Red}(\mathcal{N}_{S_N,\mathbf{v}_N}) = \mbox{\em Red}(\mathcal{N}_{\tilde{S}_N,\tilde{\mathbf{v}}_N})$  and $\mathbf{w}_{S_N,\mathbf{v}_N} = \mathbf{w}_{\tilde{S}_N,\tilde{\mathbf{v}}_N} $
\end{enumerate}
 Furthermore, if any of the equivalent conditions above holds, then  $\mathcal E_{\tilde{\nu}}\asymp\mathcal E_\nu$.  
\end{theo}

The proof of Theorem \ref{thm: main2} relies heavily on the following application of Theorem \ref{theo: GN-Dirich}, which relates probability measures on $G$ to their counterparts on the nilpotent group $N$. 
\begin{lem}
\label{lem: GN_comp}
\begin{enumerate}[(i)]
    \item  Let $\kappa_S$  be the uniform measure on $S\cup S^{-1}$. Let $\mu_{S_N,\mathbf{v}_N}$ be the symmetric probability measure on $N$ defined as in (\ref{muSigma}). 
    We have    $$\mathcal{E}_{\mu_{S, \bm{\alpha}}}(f,f) \asymp \mathcal{E}_{\mu_{S_N,\mathbf{v}_N}}(f|_N, f|_N) + \mathcal{E}_{\kappa_S}(f,f)  $$
    \item  Let $\kappa_X$ be the uniform measure on $X$. Define a measure $\rho_N$ on $N$ by
    $$\rho_{N,S,\mathbf{v}} (n) =  \|n\|_{N,S, \mathbf{v}}^{-(1+\gamma)}$$
    where $\gamma$ is the same exponent that appears in $\rho_{G,S,\mathbf{v}}$. 
    Then   $$\mathcal{E}_{\rho_{G,S,\mathbf{v}}}(f,f)\asymp \mathcal{E}_{\rho_{N,S,\mathbf{v}}}(f|_N,f|_N) + \mathcal{E}_{\kappa_X}(f,f)$$ 
\end{enumerate}
\end{lem}

\begin{proof} 
(i) This is proved by considering the different components of the measure     
$$\mu_{S,\bm{\alpha}}(g) \asymp \sum_{i=1}^k \sum_{a\in \mathbb{Z}} \frac{\mathbf{1}_{ s_i^a}(g)}{(1+|a|)^{1+\alpha_i}}$$
separately. For each $i\in\{1,\dots, k\}$, the component associated with $s_i,\alpha_i$ is responsible for several components of the measure $\mu_{S_N,\bm{\alpha}_N}$, namely, those associated with
$xs_i^{\epsilon_i}x^{-1}$, $x\in X$. This is precisely compatible with Theorem \ref{theo: GN-Dirich}. When applying this theorem, the Dirichlet forms associated with $Y$ and $X$ are controlled using $\mathcal E_{\kappa_S}$. Here $\Gamma=\langle s_i\rangle$ and $\Gamma\cap N=\langle s_i^{\epsilon_i}\rangle$ so that $Y=\{e,s_i,\dots,s_i^{\epsilon_i-1}\}$. Only those $s_i$ for which $\langle s_i\rangle\cap N$ is infinite need to be treated in this manner. The other are torsion and their contributions are absorbed by $\mathcal E_{\kappa_S}$.\\\\
(ii)  
Note that for $n \in N$, $\rho_{G,S,\mathbf{v}}(n) = \rho_{N,S,\mathbf{v}}(n)$ and 
    $$\rho_{N,S,\mathbf{v}} (n) \asymp \rho_{G,S,\mathbf{v}}(nx) \asymp \rho_{G,S,\mathbf{v}}(xn) = \rho_{G,S,\mathbf{v}}( (xnx^{-1} )x) \asymp \rho_{N,S,\mathbf{v}}(xnx^{-1}) $$
    The result follows from Theorem \ref{theo: GN-Dirich}. 
\end{proof}

 \begin{proof} (of Theorem \ref{thm: main2})
    By Theorem \ref{thm: rho_mu}, $\mathcal{E}_{\mu_{S_N,\bm{\alpha}_N}} \asymp \mathcal{E}_{\rho_{N,S,\mathbf{v}}}$ and by Lemma \ref{lem: GN_comp}, $\mathcal{E}_{\rho_\nu} \asymp \mathcal{E}_{\mu_{\nu}}$.  To prove the first comparison $\mathcal{E}_{\nu} \asymp \mathcal{E}_{\rho_{S,\mathbf{v}}}$, we'll show each summand $\mu_{S_i, \bm{\alpha}_i}$ in $\mu_{S,\bm{\alpha}}$ has Dirichlet form comparable to that of $\nu_i$. 
     Note that Theorem \ref{thm: coord}, although proved in the nilpotent setting, also extends to groups of polynomial growth and yields
    $\mathcal{E}_{\nu_i} \asymp \mathcal{E}_{\mu_{S_i, \bm{\alpha}_i}}$ if 
    $\nu_i \asymp \nu_{S_i,\bm{\alpha}_i}$. If $\nu_i \asymp \rho_{G, S_i,\mathbf{v}_i}$, again by Theorem \ref{thm: rho_mu} and Lemma \ref{lem: GN_comp},  the comparison $\mathcal{E}_{\nu_i} \asymp \mathcal{E}_{\mu_{S_i, \bm{\alpha}_i}}$ still holds, completing the proof. 
    The proof of the comparison of return probabilities is identical to the one in Lemma \ref{lem: return_prob}. 
All results giving criteria for when  $\nu$ and $\tilde{\nu}$ have comparable Dirichlet forms follow from the norm comparison in Theorem~\ref{thm: gamma_G} and Theorem~\ref{thm: main1}.
\end{proof}

\section{Constructing measures from descending series}

In light of Theorem \ref{thm: wt_series} and Theorem $\ref{thm: main1}$, in this section, we now pivot to an algebraic framework for constructing measures with comparable Dirichlet forms from a given descending series. 
Let $N$ be a nilpotent group. 

Suppose $U = \{\frac{1}{2} < u_1 < \ldots <  u_{j_* + 1} < \ldots \}$ is a semigroup under addition and  
\begin{align}
    \label{se: wt_series_plain}
     \mathcal{I}_N = \Big\{ I_{u_1} \supset \ldots \supset I_{u_m}   \supset   
 Tor(N) =  N_{u_{m+1}}   \supseteq    \{e\} =   N_{u_{j_* + 2}} = \ldots  \Big\}
\end{align}
is a central series indexed by $U$ such that 
\begin{itemize}
    \item the additive filtration property holds, i.e. $[I_{u_i}, I_{u_j}]\subseteq I_{u_i + u_j} $
    \item $I_{u_1}$ is a finite-indexed subgroup of $N$
    \item for $i=1,\ldots, m$, $I_{u_i}/I_{u_{i+1}}$ is nontrivial and torsion-free. 
\end{itemize}
Define 
$$\mathfrak{u} := \mathfrak{u}_{\mathcal{I}_N}: 
N \to \mathbb{R}, \quad \mathfrak{u}(s) =  
\begin{cases}
\infty & \text{ if }s \in \mbox{Tor}(N)\\ 
   \max_{i=1,\ldots,m}\{u_i\;:\;s \in I_{u_i}\} &\text{ otherwise}
\end{cases}
$$ 

\begin{rem}
In the previous section, all the constructions are built upon a tuple $(S,\mathbf{v})$. Now we instead fix an arbitrary descending series $\mathcal{I}$ satisfying precisely the same properties, listed in Remark \ref{rem: reduced_se}), as the reduced series $\mbox{Red}(\mathcal{N}_{S,\mathbf{v}})$
arising from such tuples. Equivalently, we can start from a descending series $\mathcal{N}$  of $N$, indexed by a semigroup of weights, that satisfies the 
additive filtration property and then apply the reduce series construction in Section \ref{sec: reduced} to obtain $\mathcal{I}$. Furthermore, in light of Remark \ref{rem: reduced_wt}, if $\mathcal{I} = \mbox{Red}(\mathcal{N}_{S,\mathbf{v}})$ for some tuple $(S,\mathbf{v})$,  $\mathfrak{u}_\mathcal{N}$ coincides with $\mathfrak{w}_{S,\mathbf{v}}$. 
\end{rem}

For $i=1,\ldots, m$, define 
$$C_{u_i} = \langle [I_{u_{n}},  I_{u_{m}}]\;:\;  u_{n} + u_{m} =  u_{i}, n,m <  i\rangle \subseteq I_{u_{i}}$$
and let $\Xi_{u_i} \subset I_{u_i}$ be the subset whose image forms a generating set of 
$$ \frac{I_{u_i}/I_{u_{i+1}}}{(C_{u_i} I_{u_{i+1}} )/I_{u_{i+1}}}$$
By construction, we have 
$$\left\langle [\sigma_{i_1}, \ldots, \sigma_{i_l}]\;:\; l \ge 1, \sigma_{i_\bullet} \in \Xi_{u_{i_\bullet}}, \sum_{j=1}^l u_{i_j} = u_i\right \rangle \equiv I_{u_i} \text{ mod }I_{u_{i+1}}$$
and furthermore, 
$$\left\langle [\sigma_{i_1}, \ldots, \sigma_{i_l}]\;:\; l \ge 1, \sigma_{i_\bullet} \in \Xi_{u_{i_\bullet}}, \sum_{j=1}^l u_{i_j} \ge  u_i\right \rangle =  I_{u_i} $$

\begin{theo}

\label{thm: Xi}
Let $\Xi_0 = (\xi_1,\ldots, \xi_n)$ be a set of $N$
such that 
$$\Xi := \left(\bigsqcup_{i=1}^{m} \Xi_{u_i}\right) \sqcup \Xi_0$$
generates $N$, and  for $i=1,\ldots, n$, take $a_i > \frac{1}{2}$ with $a_i \le \mathfrak{u}(\xi_i)$. Construct 
\begin{align*}
 \mathbf{u} &:= \left(\bigsqcup_{i=1}^{m} \underbrace{(u_i,\ldots, u_i)}_{|\Xi_{u_i}|-times}  \right) \sqcup (a_1,\ldots,a_n) 
\end{align*}
If $(S,\mathbf{v})$ 
is a tuple with $S$ generating $N$ and  $\mbox{Red}(\mathcal{N}_{S,\mathbf{v}}) = \mathcal{I}$, then 
$$\mathcal{E}_{\mu_{S,\mathbf{v}}} =  \mathcal{E}_{\mu_{\Xi, \mathbf{u}}}$$
\end{theo}

\begin{proof}
Take a non-torsion element $g \in N$.  By Remark \ref{rem: reduced_wt}, suppose
$\mathfrak{w}_{S,\mathbf{v}}(g) = u_{p}$ for some $p\in \{1,\ldots,m\}$. Then $g \in 
I_{u_p}$ can be written as a product of commutators of the form 
$$[\sigma_{i_1}, \ldots, \sigma_{i_l}] \qquad 
l\ge 1, \sigma_{i_\bullet} \in \Xi_{u_{i_\bullet}}, u_{p} \le u_{i_1} + \ldots u_{i_l} $$
It follows $\mathfrak{w}_{\Xi, \mathbf{u}}(g) \ge u_{p} =  \mathfrak{w}_{S,\mathbf{v}}(g) $.  To prove it is in fact an equality, we first define  
$$f: \Xi \to (1/2,\infty], \quad f(\sigma) = 
\begin{cases}
    u_{i} &\text{ if }\sigma \in \Xi_{u_i}, i =1,\ldots,m \\
    a_i &\text{ if } \sigma = \xi_i \in \Xi_0
\end{cases}$$
Recall by construction, $f \le \mathfrak{u}$ on 
$\Xi_0$ and $f =\mathfrak{u}$ on 
$\left(\bigsqcup_{i=1}^{m} \Xi_{u_i}\right)$. 
Suppose $\mathfrak{w}_{\Xi, \mathbf{u}}(g) = v$. This implies that some power of $g$ can written as a product of commutators of the form $[\sigma_{n_1},\ldots, \sigma_{n_m}]$, where each $\sigma_{n_\bullet} \in \Xi$  and 
 $\sum_{i=1}^m f(\sigma_{n_i}) \ge v$. 
 Since, on $\Xi$, 
 $f \le \mathfrak{u} = \mathfrak{w}_{S,\mathbf{v}}$ and the weights indexing $\mathcal{N}_{S,\mathbf{v}}$ is closed, 
 the commutator $[\sigma_{n_1},\ldots, \sigma_{n_m}]$ lies in $I(N_{\tilde{v}})$ where 
 $$\tilde{v} = \sum_{i=1}^m \mathfrak{w}_{S,\mathbf{v}}(\sigma_{n_i}) \ge \sum_{i=1}^m f(\sigma_{n_i}) \ge 
 v$$ and hence so does $g$. In other words, $\mathfrak{w}_{S, \mathbf{v}}(g) \ge  \mathfrak{w}_{\Xi, \mathbf{u}}(g)$. The rest now follows from Theorem \ref{thm: main1}. 
\end{proof}

Let $G$ be a finitely generated group of polynomial growth and $N$ be a finite-index normal nilpotent subgroup of $G$. Let $(S,\mathbf{v})$ be a tuple where $S$ generated $G$ and $\mathbf{v} \in (1/2, \infty)^{|S|}$. 
Note that when the underlying group is no longer nilpotent, in a stable-like measure induced by $(S,\mathbf{v})$, 
the component that truly drives the faster decay of the return probability is the part of the measure associated with the tuple  $(S_N, \mathbf{v}_N)$ defined as in Section \ref{sec: G_to_N}. We emphasize that $S_N$ and $\widetilde{S_N}$ (defined in Definition \ref{def: norm_N}) are different in general: $S_N$ need not generate $N$, whereas $\widetilde{S_N}$ is obtained by adjoining additional elements to $S_N$ so as to generate $N$.

As a result, in constructing measures with behavior comparable to those induced by $(S,\mathbf{v})$, it suffices to mimic the contribution coming from the component that matters, namely that determined by $S_N$. Accordingly, we focus on the reduced series $\mathrm{Red}(\mathcal{N}_{S_N,\mathbf{v}N})$. If $\mathrm{Red}(\mathcal{N}_{S_N,\mathbf{v}_N}) = \mathcal{I}_{\langle S_N\rangle}$ where $\mathcal{I}_{\langle S_N\rangle}$ is a descending series of $\langle S_N\rangle$ satisfying  the same conditions as those for Series \ref{se: wt_series_plain}, we carry out the construction in  Theorem \ref{thm: Xi} to find $(\Xi, \mathbf{u})$ with the nilpotent subgroup $N$ there identified with $\langle S_N\rangle$. In particular, $\Xi$ generates $\langle S_N\rangle$, not $N$. 

\begin{theo}
Referring to the notations and setting above, let $\Xi_1$ be a subset of $G$ such that $\Xi \sqcup \Xi_1$ generates $G$.  If 
    $\mbox{Red}(\mathcal{N}_{S_N, \mathbf{v}_N}) = \mathcal{I}$, then 
$$\mathcal{E}_{\mu_{S,\frac{1}{\mathbf{v}}}} \asymp \mathcal{E}_{\mu_{\Xi, \frac{1}{\mathbf{u}}} + \kappa_{\Xi_1}}$$
where $\kappa_{\Xi_1}$ is the uniform measure on $\Xi_1$. 
\end{theo}
 \begin{proof}
It follows directly from  Lemma \ref{lem: GN_comp} which gives
$$ \mathcal{E}_{\mu_{S, \frac{1}{\mathbf{v}}}}(f,f) \asymp \mathcal{E}_{\mu_{S_N,\frac{1}{\mathbf{v}_N}
 }}(f|_N, f|_N) + \mathcal{E}_{\kappa_S}(f,f)  $$
and  Theorem \ref{thm: Xi} with $N$ identified with $\langle S_N\rangle$ which gives
$\mathcal{E}_{\mu_{S_N,\frac{1}{\mathbf{v}_N}
 }}  \asymp  \mathcal{E}_{\mu_{\Xi, \frac{1}{\mathbf{u}} } }$. 
 \end{proof}

\begin{appendices}

\section{Miscellaneous Lemmas}

\begin{lem}
    \label{lem: floor}
    Suppose $p_1,\ldots, p_l$ are positive with
    $\sum_{i=1}^l p_i = 1$. Then for any $a\in \mathbb{Z}$ 
    $$a - \prod_{i=1}^l \lfloor a^{p_i}
    \rfloor \preccurlyeq a^{1-p_*}$$
    where $p_* = \min_i\{p_i\}$. 
\end{lem}
\begin{proof}
   Let $\{k\}$ denote the fraction part of a number $k$. Then for large $a$, 
   $$ \prod_{i=1}^l \lfloor a^{p_i}
    \rfloor =  \prod_{i=1}^l a^{p_i}  \prod_{i=1}^l \left(1 - \frac{\{a^{p_i}\}}{a^{p_i}}\right) \asymp a  \left(1 - \sum_{i=1}^l \frac{\{a^{p_i}\}}{a^{p_i}}\right) \succcurlyeq a (1- l\cdot a^{-p_*}) $$
    The desired upper bound follows. 
\end{proof}

\begin{lem}
\label{lem: decomp}
For any $h \in G$, 
$$\sum_{g\in G} |f(gh) - f(g)|^2  = \sum_{g\in G} |f(gh^{-1}) - f(g)|^2 $$
and furthermore if it can be decomposed as a product $h=\prod_{i=1}^k h_i$, $$\sum_{g\in G} |f(gh) - f(g)|^2 \le  
   k \sum_{i=1}^k \sum_{g\in G}|f(gh_i) - f(g)|^2 $$
\end{lem}
\begin{proof}
    The first equality is a result of the Cayley's theorem, i.e. $gG = G$ for every $g \in G$. The second inequality follows from Cauchy Schwartz inequality, Cayley's theorem and a telescoping sum argument. 
\end{proof}
\begin{lem}
Given $(\alpha_1,\ldots,\alpha_k)$, for  $m \le k$ and any fixed $(v_{m+1}, \ldots, v_k) \in \mathbb{Z}^{k-m}$
    \label{lem: sum}
    $$\sum_{(v_1, \ldots, v_m) \in \mathbb{Z}^{m}} \left(1+ \sum_{i=1}^k |v_i|^{\alpha_i} \right)^{-1 - \sum_{i=1}^k \frac{1}{\alpha_i}} \asymp \left(1 + \sum_{i=m+1}^k |v_i|^{\alpha_i}\right)^{-1 - \sum_{i=m+1}^k \frac{1}{\alpha_i}} $$
\end{lem}
\begin{proof}
    Set
$$
B \;:=\; 1+\sum_{i=m+1}^k |v_i|^{\alpha_i},\quad
\tau_1 \;:=\; \sum_{i=1}^k \frac1{\alpha_i},\quad 
\tau_2 \;:=\; \sum_{i=m+1}^k \frac1{\alpha_i} . 
$$
Define $\Phi: \mathbb{Z}^m \to  \mathbb{R}, \Phi(u) = \sum_{i=1}^m |u_i|^{\alpha_i}$. Then 
\[
LHS
=\sum_{u\in\mathbb Z^m}(B+\Phi(u))^{-1-\tau_1}
\asymp \int_{\mathbb R^m} (B+\Phi(x))^{-1-\tau_1}\,dx.
\]
With the change of variables $x_i=B^{1/\alpha_i}y_i$, $1\le i\le m$,
\[
\int_{\mathbb R^m} (B+\Phi(x))^{-1-\tau_1}\,dx
= B^{\tau_1-\tau_2}\int_{\mathbb R^m}\bigl(B(1+\Phi(y))\bigr)^{-1-\tau_1}\,dy
\asymp B^{-1-\tau_2}.
\]
Substituting back $B$ gives the claim.
\end{proof}

\end{appendices}

\bmhead{Funding}
Both authors' research was partially supported by NSF grants DMS-2054593 and DMS-234386.

%\bibliography{sn-bibliography}

\end{document}